\documentclass[12pt]{amsart}
\usepackage{tikz}
\usepackage{color}

\usepackage{amsmath}
\usepackage{amssymb}
\usepackage{graphicx}
\usepackage{multicol}

\usepackage{enumerate}
\usepackage{float}

\newcommand{\eps}{\varepsilon}

\newcommand{\lam}{\lambda}

\newcommand{\R}{\mathbb{R}}

\newcommand{\Z}{\mathbb{Z}}

\newcommand{\calA}{\mathcal A}
\newcommand{\calC}{\mathcal C}
\newcommand{\calAm}{(\calA, m)}
\newcommand{\calAtilde}{\widetilde \calA}
\newcommand{\DeltaA}{\Delta_{\calA}}
\newcommand{\DeltaAm}{\Delta_{(\calA,m)}}
\newcommand{\D}{\mathcal D_G}
\newcommand{\Dzero}{\mathcal D_{G,0}}
\newcommand{\Dodd}{\mathcal D_{G,\text{odd}}}
\newcommand{\Dl}{\mathcal D_{G,{\ell}}}  

\DeclareMathOperator{\fix}{Fix}

\DeclareMathOperator{\pstab}{pStab}
\DeclareMathOperator{\aut}{Aut}

\DeclareMathOperator{\dom}{dom}

\newcommand{\Rall}{V^n}

\definecolor{mycolor}{RGB}{10, 100, 20}
\newcommand{\note}[1]{{\color{mycolor}{\bf Note:} #1}}%

\theoremstyle{plain} 
\newtheorem{thm}{Theorem}[section]

\newtheorem{lem}[thm]{Lemma}
\newtheorem{prop}[thm]{Proposition}

\theoremstyle{definition}
\newtheorem{definition}[thm]{Definition}
\newtheorem{example}[thm]{Example}

\newtheorem{conj}[thm]{Conjecture}

\theoremstyle{remark}
\newtheorem{rem}[thm]{Remark}

\floatstyle{boxed}
\newfloat{Algorithm}{tbp}{lop} 
\floatname{Algorithm}{\sc{Algorithm}}

\begin{document}

\thanks{\today}

\title[Synchrony and Anti-Synchrony]
{Synchrony and Anti-Synchrony for Difference-Coupled Vector Fields on Graph Network Systems}

\author{John M. Neuberger}
\author{N\'andor Sieben}
\author{James W. Swift}

\email{
John.Neuberger@nau.edu,
Nandor.Sieben@nau.edu,
Jim.Swift@nau.edu}

\address{
Department of Mathematics and Statistics,
Northern Arizona University PO Box 5717,
Flagstaff, AZ 86011-5717, USA
}


\begin{abstract}
We define a graph network to be a coupled cell network where there are only one type of cell and one type of symmetric coupling between the cells.
For a difference-coupled vector field on a graph network system, all the cells have the same internal dynamics, and the coupling between
cells is identical, symmetric, and depends only on the difference of the states of the interacting cells.
We define four nested sets of difference-coupled vector fields by adding further restrictions on the internal dynamics and the coupling 
functions. These restrictions require that these functions preserve zero or are odd or linear. 
We characterize the synchrony and anti-synchrony subspaces with respect to these four subsets of admissible vector fields. 
Synchrony and anti-synchrony subspaces are determined by partitions and matched partitions of the cells that satisfy
certain balance conditions. We compute the lattice of synchrony and anti-synchrony subspaces for
some examples of graph networks.
We also apply our theory to systems of coupled van der Pol oscillators.
\end{abstract}


\maketitle

\tolerance=10000


\section{Introduction}

Coupled cell networks are an important object of study with diverse applications and an extensive literature.
Stewart, Golubitsky and Pivato \cite{Pivato} introduced the concept of \emph{balanced} equivalence relations to study coupled cell networks.  Their work showed that synchrony subspaces arise from 
balanced equivalence relations, demonstrating that robustly invariant subspaces exist in coupled cell networks beyond those
forced by symmetry alone.  Many papers followed extending this work, notably \cite{Dias2004, Dias2005, FieldComb, G&Sgroupoid, G&Sreview, Torok}.
Coupled cell networks have also been studied in the physics literature, where \emph{cluster synchronization} is the term used to describe the dynamics on the synchrony subspace \cite{PecoraClusterSynch, SchaubClusterSynch, SorrentinoClusterSynch}.

The internal symmetry of the cells is of course important, provided the coupling respects the symmetry \cite{Dionne2, Dionne1}. 
For example, odd cell dynamics can lead to anti-synchrony as well as synchrony, wherein some cells are
$180^{\circ}$ out of phase with others.  This has interested physicists especially as it applies to control of chaotic oscillators \cite{KimAntiSynch}.


In this paper we study a special type of coupled cell network that we call a \emph{graph network}. 
A graph network is homogeneous, so there are only one type of cell and one type of coupling between the cells. 
The coupling between a pair of cells is symmetric, so the network connections are determined by a connected simple graph. 
In a \emph{graph network system} the state of a cell is described by an element of $\R^k$.


We consider \emph {difference-coupled} vector fields that are admissible on our graph networks. In these vector fields
all the cells have the same internal dynamics, and the coupling between
cells is identical, symmetric, and depends only on the difference of the states.
The coupling function is evaluated at the difference between the states of cells joined by an edge in the graph.
Difference-coupled vector fields are found in many coupled cell systems modeling natural phenomena.
This special functional form of the coupling is what differentiates our work from \cite{Torok,Pivato} and the research that followed.


We define three strictly nested subsets of difference-coupled vector fields.
In an \emph{exo-difference-coupled} vector field the coupling function preserves zero.
The condition means that two cells in an identical state do not influence each other even if there is an edge between them.  
In an \emph{odd-difference-coupled} vector field the internal dynamics and coupling functions are both odd.
An odd coupling function means that the influence of one cell to another is the negative of the reverse influence. 
In a \emph{linear-difference-coupled} vector field the internal dynamics function is odd and the coupling function is a linear operator.


Our main goal is to characterize the subspaces of the total phase space of a graph network system that are invariant under every vector field in one of our four collections of difference-coupled vector fields. 
These invariant subspaces exhibit synchrony or anti-synchrony of the cells in the network. This means that certain cells are either in the same state, or their states are the negatives of each other. 
The significance of invariant subspaces come from the fact that invariant subspaces are also dynamically invariant. This means that if a network dynamical system is in a state of synchrony or anti-synchrony, then it remains in that state.
Synchrony and anti-synchrony spaces correspond to partitions and matched partitions of the network cells. Cells in synchrony are in the same equivalence class, and the equivalence classes that contain cells in anti-synchrony are matched.  


Our main Theorem~\ref{MainTheorem} characterizes the invariant subspaces with respect to our four subsets of admissible vector fields as synchrony or anti-synchrony spaces corresponding to certain type of partitions of the network cells. The partitions we need are balanced, exo-balanced, odd-balanced, or linear-balanced.
Balanced partitions are called equitable partitions in the physics literature, for example \cite{SchaubClusterSynch}. Similarly, exo-balanced partitions
are called external equitable partitions in
\cite{AguiarDiasWeighted, SchaubClusterSynch}.
To our knowledge, the definitions of odd-balanced and linear-balanced partitions are new. The signed equitable external partitions of \cite{SchaubClusterSynch} are different from ours.

The dynamics on an invariant synchrony subspace or anti-synchrony subspace is described
by an easily constructed reduced system on the lower dimensional invariant subspace.
This makes reduced systems a useful computational tool for efficiently finding solutions to partial difference equations (PdE, \cite{NSS3}) with given local symmetry.
The reduced system should correspond to some sort of quotient cell network system. We do not fully understand the nature of this quotient network but it is not difference-coupled.
Our situation is similar to that of \cite{Pivato} where the quotient cell network is not the same type of network as the original. We hope that our work can be generalized to resolve this issue as it was resolved in \cite{Torok}.


Requiring invariance under a smaller set of vector fields allows for more invariant subspaces. The invariant subspaces form a lattice under reverse inclusion. This lattice is an extension of the lattice of fixed-point subspaces, isomorphic to the the lattice of isotropy subgroups. The subspaces of this extension exhibit a richer structure of local symmetries. 


Finally, our results are applied to systems of coupled generalized van der Pol oscillators. The systems show synchrony and anti-synchrony that would not be expected based on symmetry alone. By the choice of the parameters we find dynamical systems with difference coupled vector fields that are in any of our four nested subsets.


Our main motivation is to understand the local symmetry structure of anomalous invariant subspaces of \cite{NSS3}, 
an essential first step in our efforts to develop a refined theory of bifurcations that break these local symmetries. 


\section{Graph Networks}

We define the type of coupled cell network that we study in this paper.

\begin{definition}
A {\em graph network} is a connected graph  $G$ consisting of a finite set $\mathcal{C}$ of \emph{cells} and a set $\mathcal{E}\subseteq\mathcal{C}\times\mathcal{C}$ of \emph{arrows} such that $(i,i)\notin\mathcal{E}$ for all $i$, and  if $(i,j)\in\mathcal{E}$ then $(j,i)\in\mathcal{E}$.
We call $ij:=\{(i,j),(j,i)\}$ an \emph{edge} of the graph. 
\end{definition}

Note that a graph network is essentially a connected simple graph, where the vertices are cells and the edges are back and forth arrows.
We usually take $\mathcal{C}=\{1,\ldots,n\}$ to be the set of cells. 
A graph network is a special case of the coupled cell network of \cite{Torok}. 
The coupled cell networks of \cite{Torok, Pivato} have more than one type of cells and edges, so they consider graphs with colored vertices and colored directed arrows. 
In our graph networks we do not allow multiple arrows and loops. Also, all our cells and arrows are equivalent. So our networks are \emph{homogeneous} in the sense of \cite{aldis2010balance}, which in their notation means $\sim_{\mathcal{C}}$ and $\sim_{\mathcal{E}}$ are trivial. 
For some authors \cite{Torok}, in a homogeneous network all cells are input equivalent, so they would only call our networks homogeneous if the degree of each cell is the same.
In graph networks, every arrow has an opposite arrow, so the adjacency matrix is symmetric. 

For cell $i \in \mathcal{C}$ we let 
$N(i) := \{ j \in \mathcal{C} \mid (j, i) \in \mathcal E\}$
denote the set of {\em neighbors} of $i$. Since our networks are homogeneous, two cells $i$ and $j$ are \emph{input equivalent} in the sense of \cite{Torok} when $N(i)$ and $N(j)$ have the same size.

For each cell $i\in \mathcal{C}$ let $V:=\mathbb{R}^k$ be the common \emph{cell phase space}. Then the \emph{total phase space} of the network is $P:=V^\mathcal{C}=\prod_{i\in\mathcal{C}}V$. 
Using the natural ordering on $\mathcal{C}$, we identify $P$ with $V^n$
and write $x=(x_1,\ldots,x_n)\in V^n$. We refer to the pair $(G,V)$ as a \emph{graph network system}.

\subsection{Difference-Coupled Vector Fields on Graph Network Systems}

We now define several subsets of the set of admissible vector fields 
\cite{Torok, Pivato} 
corresponding to a graph network system $(G,V)$.

\begin{definition}
\label{diffCoupled}
A \emph{difference-coupled vector field} on  $(G,V)$ is a vector field $f: V^n \to V^n$ such that the component functions $f_i:V^n\to V$ of $f$ for $ i \in \mathcal{C}$ are of the form
\begin{equation}
\label{gi}
f_i(x) := g(x_i) + \sum_{j\in N(i)} h(x_j - x_i) 
\end{equation}
with $g,h:V\to V$. We define the following sets of difference-coupled vector fields.
\begin{enumerate}
\item 
$\D$ is the set of difference-coupled vector fields on $(G,V)$.
\item
$\Dzero:=\{f\in\D\mid h(0)=0 \}$ is the set of {\em exo-difference-coupled vector fields}.
\item
$\Dodd:=\{f\in\D\mid g,h\text{ odd} \}$ is the set of {\em odd-difference-coupled vector fields}.
\item
$\Dl:=\{f\in\D\mid g\text{ odd}, h\text{ linear} \}$ is the set of {\em linear-difference-coupled vector fields}.
\end{enumerate}
\end{definition}

Note that $h$ is linear as an operator in the definition of $\Dl$, and hence also odd.
Thus 
\begin{equation}
\label{gContainment}
\Dl \subseteq \Dodd \subseteq \Dzero \subseteq \D \subseteq \mathcal F_G,
\end{equation}
where $\mathcal F_G$ is the set of admissible vector fields on the graph cell network defined in \cite{Torok, Pivato} as $\mathcal F_G^P$. 
We drop the $P$ from our notation since in our networks we always have $P=V^n$.

Recall that a subspace $W$ of $V^n$ is $f$-invariant if $f(W) \subseteq W$. If $\mathcal{D}\subseteq \mathcal F_G$ then we say $W$ is $\mathcal{D}$-\emph{invariant} if $W$ is $f$-invariant for all $f\in\mathcal{D}$. The main goal of this paper is to characterize the $\mathcal{D}$-invariant subspaces of $V^n$ for $\mathcal{D}\in \{\Dl,\Dodd,\Dzero,\D,\mathcal F_G\}$.

A slightly different vector field from that defined in Equation~(\ref{gi}) is called Laplacian coupled in the physics litterature.  For example, \cite[Equation (3)]{SorrentinoClusterSynch}
and \cite[Equation (13)]{SchaubClusterSynch} can be written in our notation as
\begin{equation}
\label{lapCoupled}
f_i(x) = g(x_i) + \sum_{j\in N(i)} \big (h(x_j)- h(x_i) \big ).
\end{equation}
In the case where $h$ is linear ({\em i.e.},~a linear operator), then Equations~(\ref{gi}) and (\ref{lapCoupled}) are the same, since $h(u-v) = h(u) - h(v)$.  
The case where $h$ is a constant multiple of the identity yields a coupling that can be written in terms the Laplacian matrix, as described in Example \ref{semilinear}.

\subsection{Examples of Difference-Coupled Vector Fields}

In this section we give several applications of difference-coupled vector fields.
The dimension of the cell phase space $V = \R^k$ is of primary importance. The main application is a graph network dynamical system $\dot x = f(x)$, where $f$ is a difference-coupled vector field. Applications to iterated maps are similar.

\begin{example}
\label{semilinear}
In \cite{NSS3}, we approximate zeros of $f \in \Dl$ defined by  $g(x_i) = s x_i + x_i^3$, and $h(y) = y$ on a graph network system with  $V = \R$.
That is, we approximate solutions to
$$
f_i(x):=s x_i + x_i^3 + \sum_{j \in N(i)} (x_j - x_i) = 0
$$
using Newton's method.
The so-called diffusive coupling term is the negative of the well-known graph Laplacian $L$, defined by
\begin{equation}
\label{Lap}
(Lx)_i = \sum_{j \in N(i)} (x_i - x_j).
\end{equation}
In \cite{NSS2,NSS5} we apply this methodology to discretizations of PDE of the form $ \Delta u + g(u) = 0$.
We make extensive use of invariant subspaces;
if the initial guess is in an $f$-invariant subspace $W$, then the next approximation obtained by a Newton step is also in $W$.
\end{example}

\begin{example}
An example of a graph network dynamical system $(G,\R)$ with $f\in\Dl$ is a heat equation on a graph network.  
Assume that we have an embedding of the graph such that each edge has the same length.  Each cell is an identical metal ball
that obeys Newton's law of cooling, with the ambient temperature defined to be $0$.
Each edge is an identical heat-conducting
pipe.  Assuming the pipe does not lose heat out the sides,
and the heat flow through the pipe is proportional to the temperature difference of the balls,
the temperature $x_i$ of ball $i$ satisfies
\begin{equation}
\label{heatEqn}
\dot x_i = -s x_i + \sum_{j \in N(i)} (x_j - x_i) = -s x_i - (L x)_i.
\end{equation}
with a suitable time scaling. 
This linear equation has a general solution with $n$ arbitrary constants $c_i$ that can be written in terms of the 
eigenvalues $\lam_i \geq 0$ and eigenvectors $\psi_i$ of the graph Laplacian,
\begin{equation}
\label{heatSolution}
x = \sum_{i = 1}^n c_i e^{(-s - \lam_i)t} \psi_i .
\end{equation}
If $G$ is a discretization of a region,
then the System~(\ref{heatEqn}) approximates
the heat equation $u_t = -su + \Delta u$.

The explicit solution (\ref{heatSolution}) shows that the subspace of $\R^n$ spanned by \emph {any} set of eigenvectors is invariant for the 
\emph{linear} system 
(\ref{heatEqn}).  The theory of invariant subspaces of linear operators is a highly developed field, and is not the subject of the current paper.
For \emph{nonlinear} systems the invariant subspaces are much less common, and these are described in this paper.

\end{example}

\begin{figure}
\includegraphics[scale = .15]{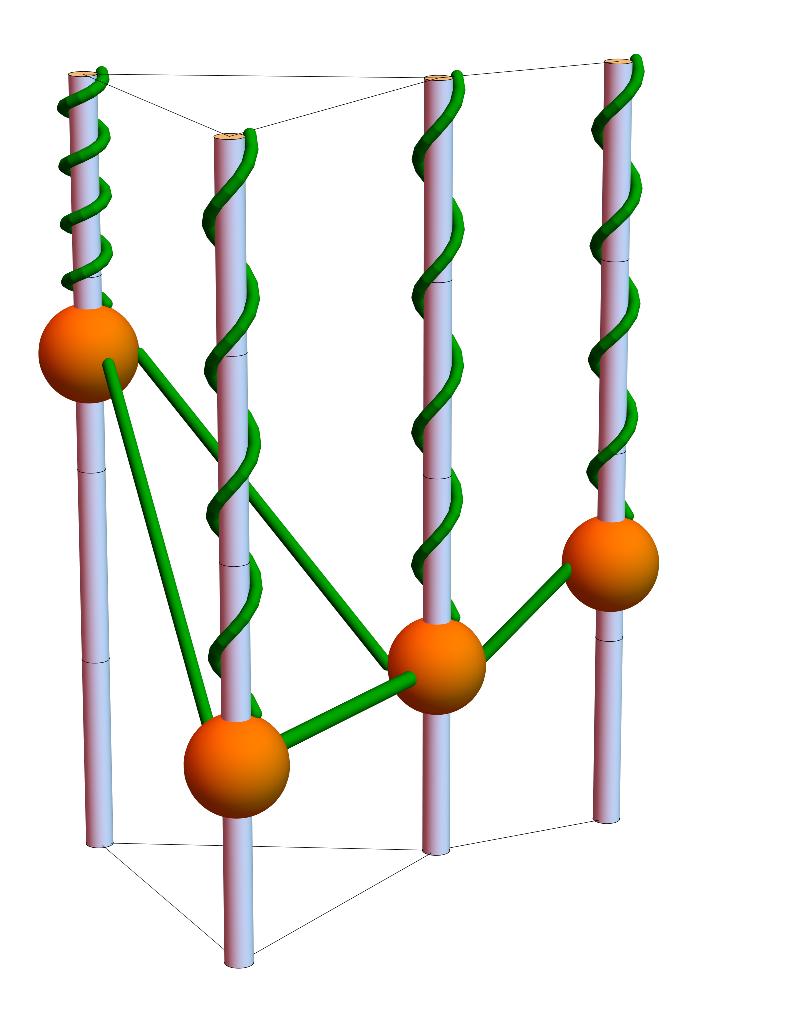}
\caption{
The mass-spring system described in Example \ref{springsEx}, based on the  paw graph.
The four masses are constrained to move vertically by the poles.
}
\label{P3springsFig}
\end{figure}

\begin{example}
\label{springsEx}
A frictionless mass-spring system $(G,\R^2)$  is shown in Figure~\ref{P3springsFig}.  
Assume that a graph $G$ can be drawn in the plane such that each edge has the same length.
The graph is drawn on the ceiling, and a mass
$m$ is suspended from each cell with a spring of spring constant $k_1$.
Each mass is constrained to move along a vertical axis.  The masses are coupled by springs with spring constant $k_2$ and natural length $0$.
First, assume that the springs are linear, obeying Hooke's law.   The dependent variable $u_i$
is the height of mass $i$ above the equilibrium position.
After scaling time using the natural frequency $\sqrt{k_1/m}$,
the system 
$$
\ddot u_i = - u_i + \sum_{j \in N(i)} \delta \, (u_j - u_i)
$$
has a dimensionless parameter $\delta = k_2/k_1 > 0$.

When each second order ODE is written as a system of two first order ODEs, the full system is an example of
Equation (\ref{gi}), where the state of each mass is described by $x_i = (u_i, \dot u_i)$.  The internal cell dynamics is given by $g(u, v) = (v, -u)$ and the coupling function is
$h(u, v) = (0, \delta u)$. 
The system is linear, so $g$ is odd, $h$ is linear, and  $f \in\Dl$. 
This system is an unrealistic model for large oscillations.
If we replace the vertical spring by a more general spring we get $g(u, v) = (v, F(u))$, where $F(u)$ is proportional to the force exerted by the spring at position $u$.  
If in addition the coupling springs are massive,
then they exert a downward force on the masses they are coupled to, and the system becomes
\begin{equation}
\label{massSpring}
\ddot u_i = F(u_i) + \sum_{j \in N(i)} \big(\gamma + \delta \, (u_j - u_i) \big),
\end{equation}
where $\gamma$ is proportional to the gravitational force due to the mass of each coupling spring.  This gives $h(u, v) = (0, \gamma + \delta u)$.

If $\gamma \neq 0$, then System (\ref{massSpring}) has $f \in \D \setminus \Dzero$.
If $\gamma = 0$ then $f \in \Dl$ for odd $F$ and $f \in \Dzero \setminus \Dodd$ for $F$ not odd.

Note that for massless (not necessarily linear) coupling springs, the coupling function $h$ is odd as a consequence of Newton's Third Law;
the spring pulls the two connected masses with equal and opposite forces.
Therefore it is easy to get a system with $f \in \Dodd\setminus\Dl$.
\end{example}

\begin{example}
\label{adjacency}
Some admissible vector fields are not difference-coupled.
Consider the graph network system $(P_3, \R)$, shown in Figure~\ref{P3Fig}.
Let $f:\R^3 \rightarrow \R^3$ be defined by $f(x) = A x$, where $A$ is the adjacency matrix of $P_3$.
It can be shown that $f$ is an admissible vector field, as defined by \cite{Pivato}, that is not 
difference-coupled.
\end{example}

\section{Partitions and Synchrony}

In this section we provide a full characterization of $\D$ and $\Dzero$-invariant subspaces of $P$. These invariant subspaces correspond to balanced and exo-balanced partitions of the cells.

\subsection{Balanced Partitions}

A partition $\calA$ of the cells of a graph network determines an equivalence relation.
For $i \in A \in \calA$ we use $[i] := A$ to denote the equivalence class of $i$.

\begin{definition}
\label{DeltaA}
Let $\calA$ be a partition of the cells of $(G,V)$. The {\em polydiagonal subspace} for $\calA$ is
$$
\Delta_{\calA} := \{x \in \Rall \mid x_i = x_{j} \mbox{ if } [i]=[j]  \}.
$$
For $x\in\Delta_\calA$ we define $x_{[i]}:=x_i$.
\end{definition}

The notation used in \cite{Torok, Pivato} and others for the polydiagonal subspace is
$\Delta_{\bowtie}$, highlighting the equivalence relation $\bowtie$ instead of the corresponding partition $\calA$.

We use capital letters like $A,B,\ldots\in\calA$ to denote equivalence classes in a partition. 
If $x\in \Delta_\calA$, then we use corresponding lower case letters $a,b,\ldots$ for the values $x_A,x_B,\ldots$, respectively.
For example, if $A=[i]$ then we write $x_{[i]}=x_A=a$.
Figures \ref{P3Fig} through \ref{C10_1_2Fig} show partitions $\calA$, or
equivalently the corresponding subspaces $\DeltaA$, for several graph networks.

We now apply the concept of a balanced equivalence relation, defined in \cite{Pivato} for coupled cell networks, to graph networks.
Note that the degree of cell $i$ is $|N(i)|$.
\begin{definition}
The {\em degree} of a cell $i$
{\em relative to a set}
 $A$ of cells is the number
$$
d_A(i) := |A \cap N(i) |
$$
of edges connecting cells in $A$ to $i$. 
\end{definition}

Our special case 
of a graph network
allows for a groupoid-free definition of balanced partitions, first defined in \cite{Pivato}
for more general coupled cell networks.
\begin{definition}
\label{balDef}
A partition $\calA$ of the cells of a graph network $G$ is \emph{balanced} if 
$$
d_A(i) = d_A(j)
$$
whenever $[i]=[j]$ and $A \in \calA$.  In this case, we define
$d_A([i]) := d_A(i)$.
A \emph{balanced subspace} is the polydiagonal subspace $\DeltaA$ for a balanced partition $\calA$.
\end{definition}

\begin{example}
\label{partitionOfSingletons}
For any graph network, the partition consisting of singleton sets is always balanced, since $[i] = [j]$ implies
$i = j$ and the condition is trivially satisfied.  An example is shown in Figure~\ref{P3Fig}(i).
\end{example}

\begin{figure}
\begin{tabular}{ccccc}
\includegraphics[scale =1]{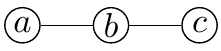} & &
\includegraphics[scale =1]{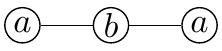} & &
\includegraphics[scale =1]{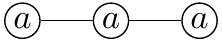} \\
(i) && (ii) && (iii)
\end{tabular}
\caption{Exo-balanced partitions of the cells of the path $P_3$.
Figures (i) and (ii) show balanced partitions.  The partition in Figure (iii) is 
strictly exo-balanced, meaning it is exo-balanced but not balanced. 
These figures can also be interpreted as showing exo-balanced subspaces.
}
\label{P3Fig}
\end{figure}

\begin{example}
\label{P3Ex}
The path $P_3$ with three cells $\mathcal{C} = \{1, 2, 3\}$ 
numbered from left to right
is shown in Figure~\ref{P3Fig}.
As mentioned in Example~\ref{partitionOfSingletons}, Figure~\ref{P3Fig}(i) shows a balanced partition.
The balanced partition $\calA = \{A, B\}$, with two classes $A = \{1, 3\}$ and $B = \{2\}$, shown in Figure~\ref{P3Fig}(ii), has degrees
$d_A(A) = 0 = d_B(B)$, $d_A(B) =2$, and $ d_B(A) = 1$.
The partition with $\calA = \{\mathcal{C}\}$, shown in Figure~\ref{P3Fig}(iii), is not balanced since
$d_\mathcal{C}(1) \neq d_\mathcal{C}(2)$.
\end{example}

\subsection{Exo-Balanced Partitions}

The following definition is the first of three 
generalizations of balanced partitions.

\begin{definition}
\label{exoDef}
A partition $\calA$ of the cells  of a graph network is \emph{exo-balanced} if 
$$
d_A(i) = d_A(j)
$$ 
whenever $[i]=[j] \neq A \in \calA$.
In this case we define $d_A([i]) := d_A(i)$ for all $[i] \neq A$. 
If a partition is exo-balanced but not balanced, we call it \emph{strictly exo-balanced}.
An \emph{exo-balanced subspace} is the polydiagonal subspace $\DeltaA$ for an exo-balanced partition $\calA$.
\end{definition}

The term exo-balanced means balanced with other equivalence classes in the partition,
but not necessarily balanced within a single equivalence class.
Clearly, every balanced partition is exo-balanced since
the definition of exo-balanced is the same as the definition of balanced with the additional condition that $[i] \neq A$.
Note that $d_A(A)$ is not defined if $\calA$ is strictly exo-balanced.

\begin{example}
The singleton partition $\calA = \{ \mathcal{C} \}$ is exo-balanced for any graph network
since the condition
for being exo-balanced is vacuously true.
This same partition is balanced if and only if every cell has the same degree.
Figure~\ref{P3Fig}(iii) shows an example of a strictly exo-balanced partition. It is not balanced since cells 1 and 3 have degree 1 
whereas cell 2 has degree 2.
\end{example}

\begin{figure}
\begin{tabular}{ccccc}
\includegraphics[scale = 1]{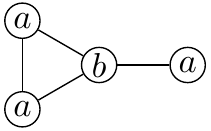} & \quad &
\includegraphics[scale = 1]{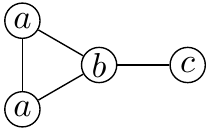} & \quad &
\includegraphics[scale = 1]{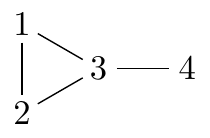} \\
(i) & & (ii) & & (iii)
\end{tabular}
\caption{Exo-balanced partitions on a four cell graph network.  
Figure (i) shows a strictly exo-balanced partition. Figure (ii) shows a balanced partition.
Figure (iii) shows the labels of the cells.}
\label{pawFig}
\end{figure}

\begin{example}
\label{pawEx}
Consider the graph network depicted in Figure~\ref{pawFig}. 
The partition 
$\calA = \{ A, B \}$, with
$A := \{1, 2, 4\}$ and $B := \{3\}$,
shown in Figure~\ref{pawFig}(i), is exo-balanced with degrees
$d_A(B) = 3$ and $d_B(A) = 1$. 
However, the partition is not  balanced, since $d_A(1) = 1 \neq d_A(4) = 0$ even though $1, 4 \in A$.
The partition in Figure~\ref{pawFig} (ii) is balanced with
$d_A(B) = 2$, $d_A(A) = d_B(A) = d_B(C) = d_C(B) = 1$, and $d_A(C) = d_C(A) = d_B(B) = d_C(C) = 0$.
\end{example}
 
Note that $d_A(B) \neq d_B(A)$ in general,
as evidenced by the partitions in Figure~\ref{P3Fig}(ii) and Figure~\ref{pawFig}. 
The following result describes the relationship between these two degrees.

\begin{prop}
\label{ABedges}
If $A$ and $B$ are different classes of an exo-balanced partition of the cells of a graph network $G$, then
$$
|A| \; d_B(A) = |B| \; d_A(B).
$$
\end{prop}

\begin{proof}
Both $|A| \; d_B(A)$ and $|B| \; d_A(B)$ count the number of edges that connect a cell in $A$ to a cell in $B$.
\end{proof}

An \emph{automorphism} of the graph $G$ is a permutation $\sigma$ of the cells that preserves the edges of $G$, that is, $ij$ is an edge exactly when $\sigma(i)\sigma(j)$ is an edge. 
The automorphisms of $G$ form a group $\aut(G)$,
which acts on $V^n$ by $\sigma \cdot x := (x_{\sigma^{-1}(1)}, x_{\sigma^{-1}(2)}, \ldots , x_{\sigma^{-1}(n)})$.

A subgroup $\Sigma$ of $\aut(G)$ is an \emph{isotropy subgroup} if $\Sigma=\pstab(\fix(\Sigma,V^n))$, where  $\fix(\Sigma, V^n)$ is the fixed point subspace of the $\Sigma$ action on $V^n$, 
and $\pstab(W)$ is the point stabilizer of $W \subseteq V^n$
\cite{NSS3}.
The fixed point subspace $\fix(\Sigma,V^n)$ of an isotropy subgroup $\Sigma$ is the polydiagonal subspace $\DeltaA$, 
where $\calA$ is the set of group orbits of the $\Sigma$  action on $\mathcal{C}$. 
The partition $\calA$ obtained this way is always balanced.
Figure~\ref{pawFig}(ii) is an example of a fixed point subspace with $\Sigma = \aut(G) \cong \Z_2$.

\begin{figure}
\begin{tabular}{ccccc}
\includegraphics{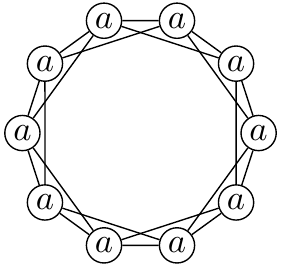} &&
\includegraphics{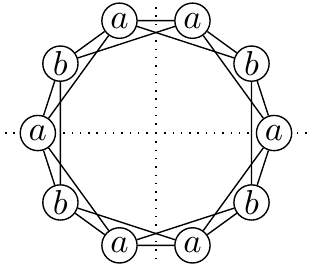} &&
\includegraphics{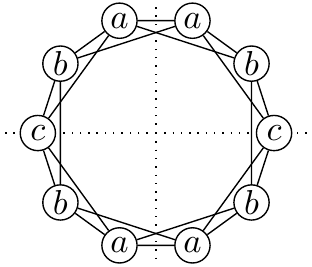} \\
(i) && (ii) && (iii)
\end{tabular}
\caption{Three balanced subspaces on the cyclic graph $C_{10(1,2)}$. This graph has 10 cells in a ring, with each cell connected to its nearest
and next-to-nearest neighbors. Figure (i) shows the fixed point subspace for $\aut(C_{10(1,2)}) \cong \mathbb{D}_{10}$.
The balanced subspace shown in (ii) is not a fixed point subspace. 
The fixed point subspace with the symmetry of the subspace in (ii) is shown in (iii).
}
\label{C10_1_2Fig}
\end{figure}

\begin{example}
While every fixed point subspace is a balanced subspace, the converse is not true.
As pointed out in \cite{Nicol}, the balanced subspace shown in Figure~\ref{C10_1_2Fig}(ii) 
is not a fixed point subspace.
The automorphism group of the graph is $\mathbb{D}_{10}$ \cite{Nicol}, and the 
point stabilizer of the balanced subspace shown is isomorphic to $\Z_2 \times \Z_2$, 
generated by the two reflections about the dotted lines.
The fixed point subspace with this symmetry is shown in Figure~\ref{C10_1_2Fig}(iii).
\end{example}

\subsection{Invariant Subspaces and Partitions}
Theorem 6.5 in \cite{Pivato}, applied to a graph network system $(G, V)$ with total phase space $P = V^n$, states that $\DeltaA$ is balanced if and only if $\DeltaA$ is 
$\mathcal F_G$-invariant.  
We present similar theorems for difference coupled vector fields.  A consequence of Theorem~\ref{balanced} is that a polydiagonal subspace $\DeltaA$ is
$\D$-invariant precisely when it is $\mathcal F_G$-invariant.  We go on to show that $\DeltaA$ is $\Dzero$-invariant if and only if $\calA$ is exo-balanced.

The proofs in this subsection make use of the simplification of Equation~(\ref{gi}) obtained by grouping neighbors of cell $i$ into equivalence classes.  Given a partition $\calA$ of cells,
\begin{equation}
\label{giProof}
f_i(x) = g(x_{[i]}) + \sum_{B \in \calA} d_B(i) h(x_B - x_{[i]})
\end{equation}
for all $x \in \DeltaA$ and all $i \in \calC$.

\begin{thm}
\label{balanced}
Let $\calA$ be a partition of the cells of $(G, V)$. Then $\DeltaA$ is $\D$-invariant if and only if $\calA$ is balanced.
\end{thm}

\begin{proof}
The backward direction of the theorem is the consequence of \cite[Theorem 6.5]{Pivato},
since $\D \subseteq \mathcal F_G$. The proof of the forward direction does not follow immediately from \cite{Pivato} since $\D \neq \mathcal F_G$.

Assume $\DeltaA$ is $\D$-invariant. Let $[i] = [j] \in \calA$.  
Using Equation~(\ref{giProof}), the $f$-invariance of $\DeltaA$ implies that
$$
f_i(x) -f_j(x) = \sum_{B \in \calA}( d_B(i) - d_B(j) ) h(x_B - x_{[i]}) = 0
$$
for all $h:V \rightarrow V$ and all $x \in \DeltaA$. Let $x\in\DeltaA$ such that $x_B-x_{[i]}$ is different for each $B\in\calA$.
For each $B \in \calA$ we can find an $h$ such that $h(x_B - x_{[i]}) \neq 0$ and $h(x_A - x_{[i]}) = 0$ for all $A \neq B$.
Thus $d_B(i) = d_B(j)$ and the partition is balanced.
\end{proof}

\begin{rem}
\label{AdjBalanced}
It was pointed out in \cite{Aguiar&Dias} that a simple linear algebra calculation can determine
if a partition is balanced. 
The subspace $\DeltaA$ is balanced if and only if 
$\DeltaA$ is invariant under the action of the adjacency matrix of the graph.
For example, the subspace shown in Figure~\ref{P3Fig}(ii) is balanced because
the general element $(a, b, a) \in \DeltaA$ satisfies
$$
\begin{bmatrix}
0 & 1 & 0 \\
1 & 0 & 1 \\
0 & 1 & 0
\end{bmatrix}
\begin{bmatrix}
a \\
b \\
a
\end{bmatrix}
=
\begin{bmatrix}
b \\
2a \\
b
\end{bmatrix}
\in \DeltaA.
$$
As mentioned in Example~\ref{adjacency}, multiplication by the adjacency matrix is
 an admissible vector field that is not a difference-coupled vector field in $\D$
on the graph $G = P_3$.
Nevertheless, the matrix can be used to find the $\D$-invariant subspaces.
\end{rem}

\begin{thm}
\label{exo-balanced}
Let $\calA$ be a partition of the cells of $(G, V)$. Then the following are equivalent:
\begin{enumerate}
\item[$(1)$]
$\DeltaA$ is $\Dzero$-invariant;
\item[$(2)$]
$\DeltaA$ is $\Dodd$-invariant;
\item[$(3)$]
$\DeltaA$ is $\Dl$-invariant;
\item[$(4)$]
$\calA$ is exo-balanced.
\end{enumerate}
\end{thm} 

\begin{proof}
We have $(1)\Rightarrow (2) \Rightarrow (3) $ since $\Dl \subseteq \Dodd \subseteq \Dzero$. 

To see that $(3)\Rightarrow (4)$, assume that $\DeltaA$ is $\Dl$-invariant.
Assume $[i]=[j]\neq A\in \calA$.
Let $g$ be the zero function and $h$ be the identity function so $f \in\Dl$.
For a fixed nonzero $v\in\mathbb{R}^k$, let $x \in \DeltaA$ be defined by
$$
x_\ell = 
\begin{cases}
v & \mbox{if }\ell \in A,\\
0 & \mbox{if }\ell \in \mathcal{C} \setminus A,
\end{cases}
$$
that is, $x_A = v$ and $x_B = 0$ if $B \neq A$. 
Invariance implies $f(x)\in\DeltaA$, so $f_i(x)=f_j(x)$.
Since $[i] \neq A$, the coupling term $h(x_B - x_{[i]})= h(0) =0$ in
Equation~(\ref{giProof}) unless $B=A$. 
Thus, we have $f_i(x) = g(x_{[i]}) + d_A(i) h(v) = d_A(i) v$ for this choice of $x$, $g$,
and $h$.  It follows that
$$
d_A(i)v = f_i(x)  = f_j(x)  = d_A(j)v,
$$
and hence $d_A(i) =d_A(j)$.
Thus $\calA$ is exo-balanced.

Finally to show that $(4)\Rightarrow (1)$, assume that $\calA$ is exo-balanced. Let $f\in\Dzero$, $x\in \DeltaA$. 
Since $\calA$ is exo-balanced, we can replace $d_B(i)$ with $d_B([i])$ in 
Equation~(\ref{giProof}), for all $[i]\ne B$. When $B = [i]$,
the term contributes nothing since $h(x_{[i]}-x_{[i]})=h(0)=0$, so
\begin{equation}
\label{giProof2}
f_i(x) = g(x_{[i]}) + \sum_{B \in \calA\setminus\{[i]\}} d_B([i]) h(x_B - x_{[i]}). 
\end{equation}
Hence $f(x)\in\DeltaA$ since $[i]=[j]$ implies $f_i(x) = f_j(x)$.

\end{proof}

\begin{rem}
\label{LapExo-balanced}
Similar to Remark~\ref{AdjBalanced}, there is an easy test for checking if a partition is exo-balanced in terms of matrix invariance.
If $g$ is zero and $h$ is the identity, 
then $f$ is multiplication with the negative of the graph Laplacian $L$ defined in Equation~(\ref{Lap}).
So in the proof of Theorem~\ref{exo-balanced}, we show that if $\DeltaA$ is $L$-invariant, 
then $\DeltaA$ is exo-balanced. The other direction is trivial.
Thus, the subspace $\DeltaA$ is exo-balanced if and only if it is $L$-invariant.

For example, the subspace in
Figure~\ref{pawFig}(i) is exo-balanced because
$$
\begin{bmatrix}
2 & -1 & -1 & 0 \\
-1 &  2 & -1 & 0 \\
-1 & -1 & 3 & -1 \\
0 & 0 & -1 & 1
\end{bmatrix}
\begin{bmatrix}
a \\
a \\
b \\
a
\end{bmatrix}
=
\begin{bmatrix}
a - b \\
a - b\\
3b - 3a \\
a - b
\end{bmatrix}
\in \DeltaA.
$$
Replacing the 
Laplacian matrix with the adjacency
matrix gives
$$
\begin{bmatrix}
0 & 1 & 1 & 0 \\
1 &  0 & 1 & 0 \\
1 & 1 & 0 & 1 \\
0 & 0 & 1 & 0
\end{bmatrix}
\begin{bmatrix}
a \\
a \\
b \\
a
\end{bmatrix}
=
\begin{bmatrix}
a + b \\
a + b \\
3a \\
b 
\end{bmatrix}
\not \in \DeltaA,
$$
so $\DeltaA$ is not balanced according to the test described in Remark~\ref{AdjBalanced}.

Similar observations were made in \cite{AguiarDiasWeighted,SchaubClusterSynch,SorrentinoClusterSynch}.   They observed
that an $L$-invariant partition is also invariant for the more general System  (\ref{lapCoupled}), and they call these
\emph{external equitable partitions} (EEP).
Even though Systems (\ref{gi}) and (\ref{lapCoupled}) are different, they coincide when $h$ is the identity function.
As a consequence, their EEPs are our exo-balanced partitions. 

In \cite{SorrentinoClusterSynch} an approach is described that uses group theoretical considerations to generate exo-balanced partions. Using their reduced search space might be more efficient than checking all possible partitions.
\end{rem}

Recall that a graph is $d$-\emph{regular} if the degree of every cell is $d$.

\begin{prop}
\label{regularNoSEB}
The cell set of a $d$-regular graph $G$ has no strictly exo-balanced partitions.
\end{prop}

\begin{proof}
The Laplacian of $G$ is $L=dI - A$, where $A$ is the adjacency matrix. 
Therefore the invariant subspaces of $L$ and $A$ are the same.  By Remarks \ref{AdjBalanced} and \ref{LapExo-balanced}, 
every exo-balanced partition is balanced.
\end{proof}

\subsection{Reduced Systems for Partitions}
\label{reducedForPartitions}
The theorems in the previous section show that the ODE $\dot x = f(x)$ can be restricted to an invariant subspace $\DeltaA$, yielding a
lower-dimensional system.  We make this formal with propositions in this section.

\begin{definition}
\label{fAdef}
If $\calA$ is a partition of $\calC$ and $f(\DeltaA) \subseteq \DeltaA$ for some $f\in\mathcal F_G$, then we define $f_{[i]} : \DeltaA \rightarrow V$ by
$$
f_{[i]}(x) := f_i(x)
$$
for all $x \in \DeltaA$ and $i \in \calC$.
\end{definition}

Note that $f_A$ is well-defined if $\DeltaA$ is $f$-invariant, since $f_i(x) = f_j(x)$ for all $i, j \in A \in \calA$.

\begin{prop}
\label{qSystemBal}
Let $\calA$ be a balanced partition of the cells of $(G, V)$, and $f \in \D$.
Then 
\begin{equation}
\label{gRestrictedBal}
f_A(x) = g(x_A) + \sum_{B \in \calA} d_B(A) h(x_B - x_A)
\end{equation}
for all $A \in \calA$, $x \in \DeltaA$. 
Thus, the restriction of $f$ to $\DeltaA$ is determined by the $|\calA|^2$ integer degrees $d_A(B)$, with $A, B \in \calA$. 
\end{prop}

\begin{proof}
Theorem~\ref{balanced} implies that $\DeltaA$ is $f$-invariant.
Equation~(\ref{gRestrictedBal}) follows from Equation~(\ref{giProof}) and the definition of $d_B(A)$ 
for balanced partitions. 
Clearly, $f  | \DeltaA$ is determined by the $f_A$ functions.
\end{proof}

There is a similar proposition for exo-balanced partitions.

\begin{prop}
\label{qSystemBal2}
Let $\calA$ be an exo-balanced partition of the cells of $(G, V)$, and $f \in \Dzero$.
Then 
\begin{equation}
\label{gRestrictedExo}
f_A(x) = g(x_A) + \sum_{B \in \calA \setminus A} d_B(A) h(x_B - x_A)
\end{equation}
for all $A \in \calA$, $x \in \DeltaA$. 
Thus, the restriction of $f$ to $\DeltaA$ is determined by the $|\calA|( |\calA| - 1)$ integer degrees $d_A(B)$,
with $A, B \in \calA$ such that $A \neq B$. 
\end{prop}

\begin{proof}
This follows from Equation (\ref{giProof2}) and the definition of $d_B(A)$ 
for exo-balanced partitions. 
\end{proof}

Now we give examples of systems of ODEs defined by pair-coupled vector fields.  This is the main use of 
Propositions~\ref{qSystemBal} and \ref{qSystemBal2}.

\begin{example}
The singleton partition $\calA = \{\mathcal{C}\}$ is always exo-balanced,  with $\DeltaA = \{(a, a, \ldots, a)\}$.
The system of ODEs $\dot x = f(x)$ with $f \in \Dzero$ restricted to $\DeltaA$ is
\begin{equation}
\label{reducedExoSys}
\dot a = g(a) .
\end{equation}
That is, solutions to this ODE in $V$ are in bijective correspondence to solutions of $\dot x = f(x)$ in $\DeltaA$.
An example of a strictly exo-balanced singleton partition is Figure~\ref{P3Fig}(iii).
If the singleton partition is balanced, as in Figure~\ref{C10_1_2Fig}(i), then $d_\calC (\calC)$ is defined, and the ODE system with $f \in \D$ is
$$\dot a = g(a) + d_\calC (\calC) h(0).
$$
\end{example}

\begin{example}
\label{abExosystems}
\begin{figure}
\begin{tabular}{ccccc}
\includegraphics[scale = 1]{figs/P3ii.pdf} & \quad &
\includegraphics[scale = 1]{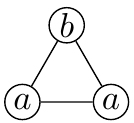} & \quad &
\includegraphics[scale = 1]{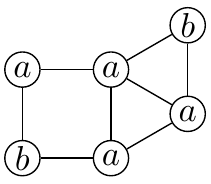} \\
(i) & & (ii) & & (iii)
\end{tabular}
\caption{
\label{3exo}
Three exo-balanced partitions with the same inter-class degrees.  The partition on the right is strictly exo-balanced.
}
\end{figure}

The three exo-balanced subspaces shown in Figure \ref{3exo} all have $d_A(B) = 2$ and $d_B(A) = 1$.  
Thus, if $f\in \Dzero$ the
restriction of $\dot x = f(x)$ to each of these exo-balanced subspaces is
\begin{align}
\begin{split}
\label{3exoSys}
\dot a & = g(a) + h(b-a) \\
\dot b & = g(b) + 2 h(a-b).
\end{split}
\end{align} 
The partition of $G = P_3$ in Figure~\ref{3exo}(i) is also balanced, with $d_A(A) = d_B(B) = 0$, 
so the restriction of the ODE system with $f \in \D$ is also Equation~(\ref{3exoSys}).
The balanced partition of $G = C_3$ in Figure~\ref{3exo}(ii) has $d_A(A) = 1$, $d_B(B) = 0$, so the restricted ODE is
\begin{align*}
\dot a &= g(a) + h(0) + h(b-a) \\
\dot b &= g(b) + 2 h(a-b)
\end{align*} 
if $f \in \D$.
The partition in Figure~\ref{3exo}(iii) is not balanced, so $\DeltaA$ is not invariant for all $f \in \D$.
\end{example}

\section{Matched Partitions and Anti-Synchrony}
When the vector field $f$ in Equation~(\ref{gi}) is odd, more invariant subspaces exist.
For example, the trivial subspace $x = 0$ is invariant for any odd $g$ and $h$.  
When $f$ is odd, the system (\ref{gi}) is equivariant under the group $\aut(G)\times \Z_2$,
where the $\Z_2$ action is $x \mapsto -x$.  The fixed point subspaces of the $\aut(G) \times \Z_2$
action are $f$-invariant for all $f \in \Dodd$ \cite{NSS3}.
For many graph networks there are invariant subspaces for all $f \in \Dodd$ that are \emph{not} fixed point subspaces.
To describe these additional
invariant subspaces, we first introduce the notion of a matched partition.

\subsection{Odd-Balanced Partitions}

\begin{definition}
An \emph{odd partition} of a finite set $\mathcal{C}$ is a set $\calA$ containing an odd number of pairwise disjoint subsets of $\mathcal{C}$ such that $\cup\calA=\mathcal{C}$. 
\end{definition}

Note that $\calA$ may contain the empty set. The number of odd partitions of $\mathcal{C}$ is the same as the number of partitions of $\mathcal{C}$. In fact, a partition with 
an even number of classes can be transformed into an odd partition by the inclusion of the empty set.

\begin{definition}
\label{matched}

A {\em matched partition} $(\calA, m)$ is an odd partition $\calA$ with a {\em matching function} $m: \calA \rightarrow \calA$
such that $m^{-1}=m$, $m$ has exactly one fixed point, and $\emptyset\in\calA$ implies $m(\emptyset)=\emptyset$. 
The fixed point of $m$ is denoted $A_0$. We use the notation $-A:= m(A)$ when the matching function is unambiguous.
\end{definition}

In other words, a matched partition $(\calA, m)$ satisfies  $-(-A) = A$ 
for all $A\in\calA$, $-A_0 = A_0$, and $-A \neq A$ for all 
$A \in \calA\setminus \{A_0\}$.  
There are two possibilities: either $A_0 = \emptyset$ or $A_0 \neq \emptyset \not \in \calA$.
Every matched partition corresponds to a subspace of $\Rall$ as follows.

\begin{definition}
\label{DeltaAmDef}
Given a matched partition $(\calA, m)$ of the cell set of a graph $G$, we define the {\em polydiagonal subspace} 
$$
\Delta_{(\calA, m)} := \{x \in \Rall \mid
x_i = - x_j \mbox{ if } [i] = - [j]
\}.
$$
For $x\in \DeltaAm$  we define $x_{[i]}:=x_i$. 
\end{definition}
\begin{rem}
The notation $x_{[i]} = x_i$ for $x \in \DeltaAm$ is well-defined since
$[i] = [i'] = -[j]$ implies $x_i = x_{i'} = -x_j$.
Note that $x_{-A} = - x_A$ motivates the definition $-A := m(A)$. 
Furthermore, if $i \in A_0$ then $x_i = 0$ since $[i] = -[i] = A_0$.
Note that $x \in \Delta_{(\calA, m)}$ is determined by its value on
just one component of each matched pair in $\calA$.
This suggests the following definition.
\end{rem}

\begin{definition}
A \emph{cross section of} a matched partition $(\calA, m)$ is a maximal subset $\calAtilde$ of $\calA\setminus\{A_0\}$ satisfying $\calAtilde \cap m(\calAtilde) = \emptyset$.
\end{definition}

In other words, a cross section $\widetilde \calA$ contains exactly one element
from each two-element subset $\{A, -A \}\subseteq\calA$, and $A_0 \not \in \calAtilde$. 
Note that $|\calAtilde|=( |\calA| - 1)/2$,
the number of cross sections is $2^{|\calAtilde|}$, and
$\dim(\Delta_{(\calA, m)})=k |\calAtilde|$.

\begin{definition}
\label{odd-balancedDef}
An \emph{odd-balanced} partition is a matched partition $(\calA, m)$ of the cell set $\mathcal{C}$ of a graph network such that for all 
$A \in \calA$ and all $i, j \in \mathcal{C}$
\begin{enumerate}
\item
$d_A(i) = d_{-A}(j)$ whenever $[i]  \neq A_0$ and $A \neq [i] = -[j]$; 
\item
$d_A(i) = d_{-A}(i)$ whenever
$[i] = A_0 \neq A$.
\end{enumerate}
In this case we define the degrees 
$d_A([i]) := d_A(i)$ 
if $[i] \neq A$ and 
$[i] \neq A_0$.
An \emph{odd-balanced subspace} is the polydiagonal subspace
$\DeltaAm$ of an odd-balanced partition $\calAm$.
\end{definition}

\begin{prop}
\label{oddBalWellDef}
Let $\calAm$ be an odd-balanced matched partition, with a cross section $\calAtilde$.  
Then $d_A(B)$ is well-defined for all $A, B \in \calA$ satisfying $A \neq B \neq A_0$, giving 
$4|\calAtilde|^2$ degrees.
Furthermore, $d_{-A}(-B) = d_{A}(B)$ whenever $d_A(B)$ is defined.
Thus all of the degrees can be determined by specifying 
$2|\calAtilde|^2$
degrees $d_A(B)$ with $A \in \calA$, $B \in \calAtilde$ , and $A \neq B$, for any 
cross section $\calAtilde$.
\end{prop}

\begin{proof}
If $A \neq [i] = [i'] = -[j] \neq A_0$, then $d_A(i) = d_{-A}(j) = d_A(i')$.  Thus
 $d_A([i])$ is well-defined 
for all $A \neq [i] \neq A_0$.
The same calculation shows that $d_A([i]) = d_{-A}(-[i])$, and thus $d_A(B) = d_{-A}(-B)$, when these are defined. The statement uses $B$ in place of $[i]$.
There are $|\calA| \cdot (|\calA| - 1)$ ordered pairs $(A, B) \in \calA \times \calA$  with $B \neq A_0$.
There are $|\calA| - 1$ pairs $(A, A)$ with $A \neq A_0$.  Thus there are  $|\calA| \cdot (|\calA| - 1) - (|\calA| - 1) = (|\calA| - 1)^2 = 4 |\calAtilde|^2$
degrees $d_A(B)$ defined.
The pairs $(A, B)$ and $(-A, -B)$ are always distinct when $d_A(B)$ is defined, since $B \neq A_0$.
Thus half of the degrees, $d_A(B)$ with $B \in \calAtilde$, determine the rest of the degrees.
\end{proof}

\begin{rem}
Note that $d_A(A)$ and $d_A(A_0)$ are not defined for odd-balanced partitions.
Figure~\ref{matchedP3-6Fig}(iii) shows an odd-balanced partition where $d_A(A)$ is not well-defined,
as described in Example~\ref{matchedP6ex}.
Figure~\ref{matchedPawFig}(ii) shows an odd-balanced partition where
$d_{A}(A_0)$
is not well-defined, as described in Example~\ref{matchedPawEx}.
\end{rem}

\begin{figure}
\begin{tabular}{ccccc}
\includegraphics[scale = 1]{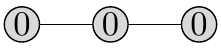} & \quad &
\includegraphics[scale = 1]{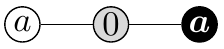} & \quad &
\includegraphics[scale = 1]{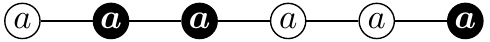} \\
(i) & & (ii) & & (iii)
\end{tabular}
\caption{Some odd-balanced partitions with a choice of cross section.  A cell  $i \in A_0$ has $x_i = 0$, 
indicated by $0$ in a gray ball.  A cell $i \not \in A_0$
is indicated
with $a$ in a white ball 
if $x_i = a$,
and $a$ in a black ball if
$x_i  = -a$.
By convention, white balls are in classes in $\calAtilde$.
For (ii) and (iii), the matched partition is $\calA = \{A, -A, A_0\}$, and the cross section is  $\calAtilde = \{ A \}$.
}
\label{matchedP3-6Fig}
\end{figure}

\begin{example}
The singleton matched partition $\{A_0\}$, with $A_0 = \mathcal{C}$ is an odd-balanced partition for any graph $G$
with cell set $\mathcal{C}$ since the conditions in Definition \ref{odd-balancedDef} are vacuously true.
The corresponding odd-balanced subspace is the $0$-dimensional subspace $\{ 0 \}$.
The cross section is $\calAtilde = \emptyset$.
An example is shown in Figure~\ref{matchedP3-6Fig}(i).  
\end{example}

\begin{example}
The matched partition of the path $P_3$ in Figure~\ref{matchedP3-6Fig}(ii) has $A = \{1\}$, $-A = \{3\}$, and $A_0 = \{2\}$.
This partition is odd-balanced: Condition (1) of Definition \ref{odd-balancedDef}
holds since $d_A(3) = d_{-A}(1)$, and Condition (2) holds since $d_A(2) = d_{-A}(2)$.
Proposition~\ref{oddBalWellDef} says that
 two degrees, $d_{-A}(A) = 0$ and $d_{A_0}(A) = 1$, determine the other degrees: $d_A(-A) = 0$ and $ d_{A_0}(-A) = 1$.
\end{example}

\begin{example}
\label{matchedP6ex}
The matched partition of the path $P_6$ in Figure~\ref{matchedP3-6Fig}(iii)
has $A = \{1, 4, 5\}$, $-A = \{2, 3, 6\}$,
and $A_0 = \emptyset$.
This partition is odd-balanced with $ d_{-A}(A) = 1$ and $d_{A_0}(A) = 0$.
Proposition~\ref{oddBalWellDef} gives $d_A(-A) = 1$ and  $ d_{A_0}(-A) = 0$.
Note that $d_A(1) = 0$ while $d_A(4) = 1$.  This shows that $d_A(A)$ is 
not well-defined for this partition.
This invariant subspace is discussed in \cite{EpsteinGolubitsky} and is a discrete
analog of the hidden symmetry described for ODEs in \cite{Warwick_1991}.
Note that this invariant subspace is not a fixed point subspace of the 
$\aut(P_6) \times \Z_2$ action on the space of functions defined on the cells of $P_6$ \cite{NSS3}.
\end{example}

\begin{figure}
\begin{tabular}{ccccc}
\includegraphics[scale = 1]{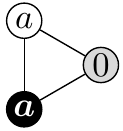} & \quad &
\includegraphics[scale = 1]{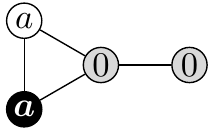} & \quad & 
\includegraphics[scale = 1]{figs/pawA.pdf}
\\
(i) & & (ii) & & (iii) 
\end{tabular}
\caption{
Two odd-balanced partitions  with the same set of degrees defined in \ref{odd-balancedDef}.  Figure (ii) shows why
we do not define $d_A(A_0)$.}
\label{matchedPawFig}
\end{figure}

\begin{example}
\label{matchedPawEx}
The two odd-balanced subspaces shown in Figure~\ref{matchedPawFig} have the same set of degrees,
$d_{-A}(A) = d_{A_0}(A) = 1$ and the two following from Proposition~\ref{oddBalWellDef}.
For the graph network in Figure~\ref{matchedPawFig}(ii), $d_A(3) =1$ and $d_A(4)=0$. 
Thus, $d_A(i)$ is not constant for all $i \in A_0$.  
This example shows that $d_A(A_0)$ is not well-defined in general.
\end{example}

\begin{figure}
\begin{tabular}{ccccc}
\includegraphics[scale = 1]{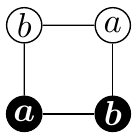} & \quad &
\includegraphics[scale = 1]{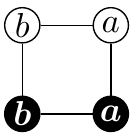}  & \quad &
\includegraphics[scale = 1]{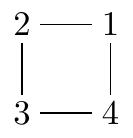} \\
(i) && (ii) & & (iii)
\end{tabular}
\caption{Two odd-balanced matched partitions $(\calA, m_1)$ and $(\calA, m_2)$ with the same partition
but a different matching function. }
\label{matchedC4Fig}
\end{figure}

\begin{example}
The two odd-balanced subspaces shown in Figure~\ref{matchedC4Fig} have the same partition
$\calA = \{ \{1\}, \{2\}, \{3\}, \{4\}, \emptyset \}$, but different matching functions.
 In both cases, $A = \{1\}$ and $B = \{2\}$.
The matching in Figure (i) is defined by $m_1(\{1\}) = \{3\}$,
$m_1(\{2\}) = \{4\}$ and $m_1(\emptyset) = \emptyset$. 
The matching in Figure (ii) is defined by $m_2(\{1\}) = \{4\}$,
$m_2(\{2\}) = \{3\}$ and $m_2(\emptyset) = \emptyset$.
\end{example}

\begin{thm}
\label{odd-balancedThm}
Let $G = (\mathcal{C}, \mathcal E)$ be a graph network, and $(\calA, m)$ be a matched partition of $\mathcal{C}$. 
The subspace $\DeltaAm$ is $\Dodd$-invariant
if and only if $(\calA, m)$ is odd-balanced.
\end{thm}

\begin{proof}
First, we develop some formulas assuming $x\in\Delta_{(\mathcal{A},m)}$
and $f\in \Dodd$. Equation~(\ref{giProof}) becomes
\begin{equation}
f_{i}(x)=g(x_{[i]})+\sum_{B\in\mathcal{A}\setminus\{[i]\}}d_{B}(i)h(x_{B}-x_{[i]})\label{giNew}
\end{equation}
since $h(0)=0$, and the term with $B=[i]$ contributes nothing.

If $[j]=-[i]$, then 
\[
f_{j}(x)=-g(x_{[i]})-\sum_{B\in\mathcal{A}\setminus\{[i]\}}d_{-B}(j)h(x_{B}-x_{[i]})
\]
since $x_{[j]}=-x_{[i]}$, $x_{-B}=-x_{B}$, and $g$ and $h$
are odd. Thus, adding the equations we get 
\begin{equation}
f_{i}(x)+f_{j}(x)=\sum_{B\in\mathcal{A}\setminus\{[i]\}}\bigl(d_{B}(i)-d_{-B}(j)\bigr)h(x_{B}-x_{[i]}).\label{giNotA0}
\end{equation}

If $[i]=A_{0}$ then $x_{[i]}=0$, so combining pairs of terms in
Equation (\ref{giNew}) results in 
\begin{equation}
f_{i}(x)=\sum_{B\in\tilde{\mathcal{A}}}\bigl(d_{B}(i)-d_{-B}(i)\bigr)h(x_{B}),\label{giA0}
\end{equation}
where $\tilde{\mathcal{A}}$ is a cross section.

For the backward direction of the proof, assume that $(\mathcal{A},m)$
is odd-balanced, $f\in\Dodd$, and $x\in\Delta_{(\mathcal{A},m)}$.
We show that $f(x)\in\Delta_{(\mathcal{A},m)}$ by verifying that
$f_{i}(x)+f_{j}(x)=0$ whenever $[i]=-[j]$. If $i\not\in A_{0}$,
then every term in the sum of Equation~(\ref{giNotA0}) is zero,
while if $i\in A_{0}$, then every term in the sum of Equation ~(\ref{giA0})
is zero.

Conversely, assume that $\Delta_{(\mathcal{A},m)}$ is $f$-invariant
for all $f\in\Dodd$. We must show that $(\mathcal{A},m)$
is odd-balanced. Fix a nonzero $v\in\mathbb{R}^{k}$. To verify Condition
(1) of Definition \ref{odd-balancedDef}, assume that $A\ne[i]=-[j]\ne A_{0}$.
We consider several cases based on the choice of $A$. In each case
we use the $f$-invariance of a carefully chosen $x\in\Delta_{(\mathcal{A},m)}$
with an appropriate $h$. To define $x$ we only need to determine
the value of $x_{B}$ for $B$ in a cross section $\tilde{\mathcal{A}}$.
We want $\tilde{\mathcal{A}}$ to always contain $[i]$ and contain
$A$ whenever possible.  
A cross section never contains $A_0$, and it cannot contain both $[i]$ and $-[i] = [j]$.
That is,
\[
\{A, [i]\} \setminus\{A_0, [j] \}\subseteq\tilde{\mathcal{A}}.
\]
Such $\tilde{A}$ always exists. With the choice of 
\[
x_{B}:=\begin{cases}
-2v & \mbox{if }B=[i]\\
\; -v & \mbox{if } B\in\tilde{\mathcal{A}}\setminus\{A, [i]\}\\
\; \, \, 3v & \mbox{if } B\in\{A\}\setminus\{A_0, [j] \}
\end{cases}
\]
and the fact that $x_{-B} = -x_B$, Equation~(\ref{giNotA0}) becomes
\begin{align}
0= & f_{i}(x)+f_{j}(x)\nonumber \\
= & \bigl(d_{A_{0}}(i)-d_{A_{0}}(j)\bigr)h(2v)+\bigl(d_{[j]}(i)-d_{[i]}(j)\bigr)h(4v)
\label{eq:long}\\
 & +\sum_{B\in\{A\}\setminus\{A_0, [j]\}}
 \left(\bigl(d_{A}(i)-d_{-A}(j)\bigr)h(5v)-\bigl(d_{-A}(i)-d_{A}(j)\bigr)h(v)\right)
 \nonumber \\
 & +\sum_{B\in\tilde{\mathcal{A}}\setminus\{[i],A\}}\left(\bigl(d_{B}(i)-d_{-B}(j)\bigr)h(v)+\bigl(d_{-B}(i)-d_{B}(j)\bigr)h(3v)\right).
 \nonumber 
\end{align}

First we verify Condition (1) of Definition \ref{odd-balancedDef}.
If $A=A_{0}$, then we choose an odd $h$ such that $h(2v)\neq0$
but $h(\ell v)=0$ for $\ell\in\{1,3,4,5\}$.   
Note that in this case and the next, the first sum in Equation~(\ref{eq:long}) has no terms.
If $A=[j]$, then we choose an odd $h$ such that $h(4v)\neq0$
but $h(\ell v)=0$ for $\ell\in\{1,2,3,5\}$.
If $A\in\mathcal{A}\setminus\{A_{0},[i], [j]\}$, then we choose
an odd $h$ such that $h(5v)\neq0$ but $h(\ell v)=0$
for $\ell\in\{1,2,3,4\}$.  
In each case, Equation (\ref{eq:long}) simplifies to $d_{A}(i)=d_{-A}(j)$.

To verify Condition (2) of Definition \ref{odd-balancedDef}, assume
that $[i]=A_{0}$ and $A\in\mathcal{A} \setminus \{A_0\}$. Let $x\in\Delta_{(\mathcal{A},m)}$
satisfy $x_B = 0$ if $B \in \calA \setminus \{A, -A\}$.
Equation~(\ref{giA0}) implies 
\[
0=f_{i}(x)=\bigl(d_{A}(i)-d_{-A}(i)\bigr)h(x_{A}).
\]
Choose $x_A$ and $h$ such that $h(x_{A}) \ne 0$. Hence $d_{A}(i)=d_{-A}(i)$. 
\end{proof}

\begin{rem}
\label{OddBalancedRem}
A linear algebra calculation similar to that described in
Remarks~\ref{AdjBalanced} and \ref{LapExo-balanced} is not apparent in this case,
because a nonlinear $f$ is needed to show that $\calAm$ is
odd-balanced in the backward direction of the proof.
\end{rem}

\subsection{Linear-Balanced Partitions}

Recall that $f\in \Dl$ when $g$ is odd and the coupling function $h$ is linear (an hence odd).
For some graph networks, there are subspaces that are invariant for all $f \in \Dl$ but not for all $f \in \Dodd$.
When $f \in \Dl$ and $x \in \DeltaAm$,
then the pair-coupled system becomes
\begin{align}
\nonumber
f_i(x)  & = g(x_{[i]}) + \sum_{B \in \calA\setminus \{[i]\}} d_B(i) \left ( h(x_B) - h (x_{[i]})  \right )
\\ & \nonumber
= g(x_{[i]}) - 2 d_{-[i]}(i) h(x_{[i]}) + \sum_{B\in \calA \setminus \{[i], -[i]\}} d_B(i) 
\left (h(x_B) - h(x_{[i]}) \right )
\\ &\label{giLinear}
= g(x_{[i]}) -  \left  (2 d_{-[i]}(i) + \sum_{B \in \calA \setminus \{[i], -[i] \}} d_B(i) \right )
h(x_{[i]} ) + 
\sum_{B\in \calA \setminus \{[i], -[i]\}} d_B(i) h(x_B).
\end{align}
The coefficient of $h(x_{[i]})$ in Equation (\ref{giLinear}) is almost the degree of $i$, except
edges coming in from $[i]$ do not count, and edges coming in from $-[i]$ count twice.
Equation (\ref{giLinear}) motivates the following.

\begin{definition}
\label{exoDeltaDef}
Given a matched partition $(\calA, m)$ of the cell set $\mathcal{C}$ of a graph network,
we define the 
\emph{linear degree} of a cell $i$ in $\mathcal{C}\setminus A_0$ to be 
\begin{equation}
\label{linearDegreeDef}
e(i) :=2 d_{-[i]}(i) + \sum_{B \in \calA \setminus \{[i], -[i] \}} d_B(i)
\end{equation}
and the \emph{degree difference} of $i\in \mathcal{C}$ with $A \in \calA$ to be
$$
\delta_A(i) = d_A(i) - d_{-A}(i).
$$
\end{definition}

Note that for $i$ in $\mathcal{C}\setminus A_0$ the linear degree is $e(i) = |N(i)| - \delta_{[i]}(i)$.
Also note that $\delta_{A_0}(i)=0$ for all $i$ because $A_0=-A_0$.
With these definitions, we can write Equation (\ref{giLinear}) compactly.
Suppose $\calAtilde$ is a cross section of the matched partition $(\calA, m)$, $i \in \mathcal{C}$,
$f \in \Dl$, and $x \in \DeltaAm$.
Then
\begin{equation}
\label{giLin}
f_i(x) =  g(x_{[i]}) -  e(i) h(x_{[i]}) +
\sum_{B \in \calAtilde \setminus \{[i]\} } \delta_B(i) h (x_B).
\end{equation}

\begin{definition}
\label{linBalDef}
A {\em linear-balanced partition} is a
matched partition $(\calA, m)$  such that for all $i,j\in \mathcal{C}$ and all $A\in \calA$
\begin{enumerate}
\item 
$e(i) = e(j)$ whenever $[i] = -[j] \neq A_0$;
\item
$\delta_A(i) = -\delta_A(j)$ whenever $[i] = -[j] $ and $[i] \neq A \ne [j] $.
\end{enumerate}
In this case we define 
$e([i]) := e(i)$
if $[i]  \neq A_0$, and 
$\delta_A([i]) := \delta_A(i)$ 
if 
$[i] \neq A \neq -[i]$.
We call a matched partition \emph{strictly linear-balanced} provided it is linear-balanced but
not odd-balanced.
A \emph{linear-balanced subspace} is the polydiagonal subspace
$\DeltaAm$ of a linear-balanced partition $\calAm$.
\end{definition}

Note that $e([i])$ and $\delta_A([i])$ are well-defined because
$e(i) = e(i') = e(j)$ whenever $[i] = [i'] = -[j] \neq A_0$, 
and $\delta_A(i) = \delta_A(i') = - \delta_A(j)$ 
whenever $[i] = [i'] = -[j] $, and $[i] \neq A \neq [j]$.
Trivial cases are excluded because a matched partition cannot have $[i] \neq A_0$ and $-[i] = A_0$.

For a choice of cross section $\calAtilde$ of a linear-balanced partition, 
the values of $e(A)$ for all $A \in \calAtilde$, and $\delta_A(B)$ for all $A \ne B$ in $\calAtilde$ determine all of 
the other linear degrees and degree differences because of the identities
$e(A) = e(-A)$ and $\delta_A(B) = - \delta_{-A}(B) = \delta_{-A}(-B) = - \delta_A(-B)$.

\begin{prop}
\label{odd2lin}
If $\calAm$ is odd balanced, then $\calAm$ is linear-balanced. 
\end{prop}

\begin{proof}
Definition \ref{exoDeltaDef} implies that
\[
e(A) = 2 d_{-A}(A) + \sum_{B \in \calA \setminus \{A, -A\} } d_B(A)
\]
for all  $A\in\calA\setminus\{A_0\}$, and 
\[
\delta_A(B) = d_A(B) - d_{-A}(B)
\]
for all $A, B \in \calA$ satisfying $A \neq B$ and $A \neq -B$.
Hence $e(A) = e(-A)$ and $\delta_A(B) = - \delta_A(-B)$ since $d_A(B) = d_{-A}(-B)$ by Definition~\ref{odd-balancedDef}.
\end{proof}

\begin{figure}
\begin{tabular}{ccccccc}
\includegraphics{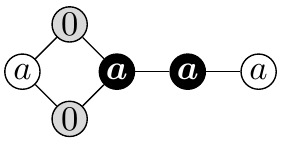}
   & \quad &
\includegraphics{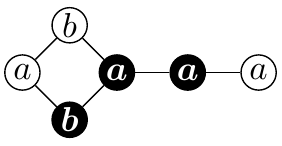}
 & \quad &
\includegraphics{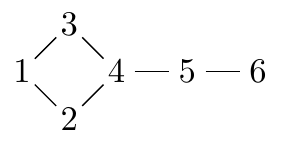}
  \\
(i) && (ii) && (iii)
\end{tabular}
\caption{Strictly linear-balanced partitions. Figure (i) has $e(A) = 2$ and
Figure (ii) has $e(A) = e(B) = 2$ and $\delta_A(B) = \delta_B(A) = 0$.
See Example~\ref{tadpoleEx}.
}
\label{tadpoleFig}
\end{figure}

\begin{example}
\label{tadpoleEx}
Consider the graph network with cell labels shown in Figure~\ref{tadpoleFig}(iii). 
Figures~\ref{tadpoleFig}(i) and \ref{tadpoleFig}(ii) show strictly linear-balanced subspaces.
First consider Figure \ref{tadpoleFig}(i).
The linear degrees satisfy $e(i)=2$ for $i\in\{1, 4, 5,  6\}$, so Condition (1) of Definition \ref{linBalDef} holds.
It is easy to check that Condition (2) holds; for example $\delta_A(2) = 0 = - \delta_A(3)$. 
The partition is not odd-balanced since $A\not=[4]= -[6]$ and $d_{A}(4) = 0 \neq 1 = d_{-A}(6)$, violating Condition~(1) of Definition~\ref{odd-balancedDef}.

Now consider Figure~\ref{tadpoleFig}(ii). Every cell has linear degree 2, so Condition~(1) of Definition~\ref{linBalDef} holds.
Condition (2) is met since $\delta_B(i) = 0$ for $i \in \{1, 4, 5, 6\}$
and $\delta_A(i) = 0$ for $i \in \{2, 3\}$.
This partition is not odd-balanced since $1 = d_B(1) \neq d_{-B}(5) = 0$ and $[1]=-[5]$, violating Condition (1) of Definition~\ref{odd-balancedDef}.
\end{example}

\begin{figure}
\begin{tabular}{ccc}
\includegraphics[scale = 1]{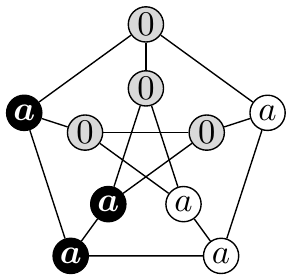} 
   & \quad &
\includegraphics[scale = 1]{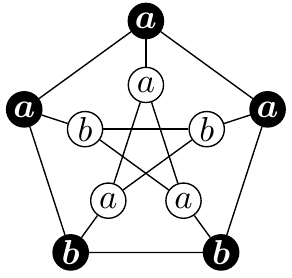}
  \\
(i) && (ii) 
\end{tabular}
\caption{
Two strictly linear-balanced partitions of the Petersen graph.
For Figure (i), the linear degree is $e(A) = 2$.
For Figure (ii), the linear degrees are $e(A) = e(B) = 2$, the degree differences are
$\delta_A(B) = \delta_B(A) = 0$.
See Example~\ref{PetersenEx}.
}
\label{PetersenFig}
\end{figure}

\begin{example}
\label{PetersenEx}
An exhaustive computer search shows that
the Petersen graph has exactly two $\aut(G) \times \Z_2$ orbits of
strictly linear-balanced partitions, shown in 
Figure~\ref{PetersenFig}. 
The subspace in Figure (i) is not odd-balanced because some cells in
$A$ have two neighbors in $A_0$ while other
cells in $A$ have no neighbors in $A_0$. 
Thus, $d_{A_0}(i)$ is not constant for all $i \in A$,
and Condition (1) of Definition~\ref{odd-balancedDef} does not hold.

The Petersen graph is 3-regular,
so Proposition~\ref{regularNoSEB} states there are no strictly exo-balanced partitions.
Our attempts to find a result analogous to Proposition~\ref{regularNoSEB},
ruling out strictly linear-balanced partitions, fail on the Petersen graph.
\end{example}

\subsection{Invariant Subspaces and Matched Partitions}

\begin{thm}
\label{gUlThm}
Let $(\calA,m)$ be a matched partition of the cells of $(G, V)$.
The subspace $\DeltaAm$ is $\Dl$-invariant if and only if $(\calA, m)$ is linear-balanced.
\end{thm}

\begin{proof}
If $x \in \DeltaAm$, $i, j \in \mathcal{C}$ satisfy $[i] = -[j]$, and $f \in \Dl$, then Equation~(\ref{giLin}) implies that
\begin{equation}
\label{gigjLin}
f_i(x) + f_j(x) =  -  \big ( e(i) - e(j) \big ) h(x_{[i]}) +
\sum_{B \in \calAtilde \setminus \{[i]\} } \big(\delta_B(i)+\delta_B(j)\big) h (x_B)
\end{equation}
for each cross section $\calAtilde$.

For the backward direction, assume $\calAm$ is linear-balanced, $f\in \Dl$, 
$x\in \DeltaAm$, and $[i] = -[j]$.
Then 
$e(i) - e(j) = 0$ if $i \not\in A_0$ and
$h(x_{[i]}) = 0$ if $i \in A_0$. 
Also $\delta_B(i) + \delta_B(j) = 0$ for all $B \in  \calAtilde \setminus \{[i]\}$.  
Equation~(\ref{gigjLin}) shows that $f_i(x) + f_j(x) = 0$ for all $[i]=-[j]$, which means $f(x)\in \DeltaAm$.

For the forward direction, assume $\DeltaAm$ is $f$-invariant for all  $f \in \Dl$.
Let $f\in \Dl$ such that $h$ is the identity function. 
To verify Condition (1) of Definition \ref{linBalDef}, assume $[i]=-[j] \neq A_0$, 
and choose an $x\in \DeltaAm$ such that $x_{[i]} \neq 0$ and $x_A = 0$ for all $A \in \calA \setminus \{[i], [j]\}$.
Since $f(x)\in \DeltaAm$ and the sum in Equation~(\ref{gigjLin}) contributes nothing, 
\[
0 = f_i(x) + f_j(x) = -\big ( e(i) - e(j) \big ) h(x_{[i]}).
\]  
Since $h(x_{[i]})=x_{[i]}\ne 0$, $e(i) = e(j)$.
To verify Condition (2)  of Definition \ref{linBalDef}, assume $[i]=-[j]$ and $A_0\ne A \in \calA$ satisfies $[i] \neq A \neq [j]$. 
Choose $x\in\DeltaAm$ such that $x_A \neq 0$ and $x_B = 0$ if $B \in \calA \setminus \{A, -A\}$.
Since $f(x)\in \DeltaAm$, $x_{[i]}=0$ and only one term ($B=A$ or $B=-A$) of the sum in Equation~(\ref{gigjLin}) survives,
\[
0 = f_i(x) + f_j(x) = \big(\delta_A(i)+\delta_A(j) \big) h (x_A).
\]
Since $h(x_A) \neq 0$, we have $\delta_A(i)= -\delta_A(j)$.
Condition (2) for the $A=A_0$ case holds trivially.
\end{proof}

\begin{rem}
\label{Linvariant}
Since the forward direction of the proof uses the identity function as $h$, we can
use a linear algebra test similar to that of Remark~\ref{LapExo-balanced}.
A matched partition $\calAm$ is linear-balanced if and only if $\DeltaAm$ is $L$-invariant.
This invariance is often easier to check than the conditions of Definition~\ref{linBalDef}. 
If the partition is linear-balanced, Definition~\ref{odd-balancedDef}
must be used to see if the partition is odd-balanced.
If the linear algebra test shows that the partition is not linear-balanced, then it is not
odd-balanced either.
For example, the subspace $\DeltaAm$ shown in Figure~\ref{matchedPawFig}(ii) is linear-balanced because
$$
\begin{bmatrix}
2 & -1 & -1 & 0 \\
-1 &  2 & -1 & 0 \\
-1 & -1 & 3 & -1 \\
0 & 0 & -1 & 1
\end{bmatrix}
\begin{bmatrix}
a \\
-a \\
0 \\
0
\end{bmatrix}
=
\begin{bmatrix}
3a \\
-3a \\
0 \\
0
\end{bmatrix}
\in \DeltaAm.
$$
As a second example, consider the path $P_3$.  The subspace $\DeltaAm = \{ (a, -a, 0) \mid a \in \R^k\}$ is not linear-balanced,
and hence not odd-balanced, because
$$
\begin{bmatrix}
1 & -1 & 0 \\
-1 & 2 & -1 \\
0 & -1 & 1
\end{bmatrix}
\begin{bmatrix}
a \\
-a \\
0
\end{bmatrix}
=
\begin{bmatrix}
2a \\
-3a \\
a
\end{bmatrix}
\not \in \DeltaAm.
$$
\end{rem}
Remarks~\ref{LapExo-balanced} and \ref{AdjBalanced} state that exo-balanced subspaces are $L$-invariant, 
and balanced subspaces are invariant under the adjacency matrix $M$.
Since linear-balanced subspaces are $L$-invariant, one might expect odd-balanced subspaces to be $M$-invariant, but
this is not the case.  For example, the odd-balanced  subspace in Figure~\ref{matchedP3-6Fig}(iii) is not $M$-invariant.  
See also Remark \ref{OddBalancedRem}.

\subsection{Reduced Systems for Matched Partitions}
\label{reducedForMatchedPartitions}
The theorems in the previous sections show that the ODE $\dot x = f(x)$ can be restricted to an invariant subspace $\DeltaAm$, yielding a
lower-dimensional system.  As in Section~\ref{reducedForPartitions}, we make this formal with propositions in this section.

We extend Definition~\ref{fAdef} to matched partitions.
\begin{definition}
\label{fAmdef}
If $\calAm$ is a matched partition of $\calC$ and  $f(\DeltaAm) \subseteq \DeltaAm$  for some $f\in\mathcal F_G$, then we define $f_{[i]} : \DeltaAm \rightarrow V$ by
$$
f_{[i]}(x) := f_i(x)
$$
for all $x \in \DeltaAm$ and $i \in \calC$.
\end{definition}
Note that $f_A$ is well-defined if $\DeltaAm$ is $f$-invariant, since $f_i(x) = f_j(x)$ for all $i, j \in A \in \calA$.
As in Proposition~\ref{qSystemBal}, we can restrict $f \in \Dodd$ to an odd-balanced partition.

\begin{prop}
\label{gUoddRestricted}
Let $(\calA,m)$ be an odd-balanced partition of the cells of $(G, V)$, 
$\calAtilde$ be a cross section,
and $f \in \Dodd$. Then
\begin{align}
\label{reducedEquationOdd}
f_A(x) = & g(x_A) - d_{A_0}(A) h(x_A) - d_{-A}(A)h(2 x_A)  \\
\nonumber
& + \sum_{B \in \calAtilde \setminus \{A\}} \big( d_B(A) h(x_B - x_A) - d_{-B}(A) h(x_B + x_A) \big)
\end{align}
for all $x \in \DeltaAm$ and all $A \in \calAtilde$. 
\end{prop}
Even though $h(0) = 0$ for $f \in \Dodd$, there are two sources of self-coupling in 
Equation~(\ref{reducedEquationOdd});  the first coupling term involving edges in the graph network
between $A_0$ and $A$,
and the second involves edges between $-A$ and $A$.

For linear coupling, many of the coupling terms cancel or combine.

\begin{prop}
\label{gUlRestricted}
Let $(\calA,m)$ be a linear-balanced partition of the cells of $(G, V)$,
$\calAtilde$ be a cross section,
and $f \in \Dl$. Then
\begin{equation}
\label{reducedUlEquation}
f_A(x) =  g(x_A) - e(A) h(x_A)
+ \sum_{B \in \calAtilde \setminus \{A\}} \delta_B(A) h(x_B )
\end{equation}
for all $x \in \DeltaAm$ and all $A \in \calAtilde$. 
\end{prop}

Note that the linearity of $h$ can be used to evaluate $h$
just one time for each $A \in \calAtilde$.
That is, Equation~(\ref{reducedUlEquation}) can be written as
$$
f_A(x) =
g(x_A) + h \left (-  e(A) x_A + \sum_{B \in \calAtilde \setminus \{A\}} 
\delta_B(A) x_B \right ).
$$

\begin{example}
When a linear-balanced partition has three elements, $\calA= \{A_0, A, -A\}$, then the dynamics
of $\dot x = f(x)$ on $\DeltaAm$ is fairly simple. If $\calAm$ is odd-balanced and $f \in \Dodd$, then 
the restriction is
$$
\dot a = g(a) - d_{A_0}(A) h(a) - d_{-A}(A) h(2a).
$$
If $\calAm$ is linear-balanced and $f \in \Dl$, then 
the restriction is
$$
\dot a = g(a) - e(A) h(a).
$$

We give a few examples here.  For the odd-balanced subspace 
shown in Figure~\ref{matchedP3-6Fig} (ii), 
the system of ODEs $\dot x = f(x)$ for $f \in \Dodd$, restricted to $\DeltaAm$, is
$$
\dot a = g(a) - h(a).
$$
The dynamics of $\dot x = f(x)$ with $f \in \Dodd$, restricted
to the odd-balanced subspace shown in Figure~\ref{matchedP3-6Fig} (iii), is 
$$
\dot a = g(a) - h(2a).
$$
If $f \in \Dl$, then the dynamics on this invariant subspace as well as the strictly linear-balanced subspaces
in Figure~\ref{tadpoleFig}(i) and Figure~\ref{PetersenFig}(i) reduce to
$$
\dot a = g(a) - 2h(a).
$$
\end{example}

\begin{example}
\label{C4oddBalanced}
The odd-balanced partition of the cells in Figure~\ref{matchedC4Fig}(i) is invariant under the ODE $\dot x = f(x)$
when $f \in \Dodd$.  The reduced dynamics of this system on $\DeltaAm$ are
\begin{align}
\begin{split}
\label{C4m1Qsystem}
\dot a = & g(a) + h(b - a) - h(b + a) \\
\dot b = & g(b) + h(a - b) - h(a + b) .
\end{split}
\end{align}
If $f \in \Dl$, the system reduces to two uncoupled, identical sub-systems.
\begin{align}
\begin{split}
\label{C4m1Qsystem}
\dot a = & g(a) -2 h(a) \\
\dot b = & g(b) -2 h(b) .
\end{split}
\end{align}
The two strictly linear-balanced partitions Figure~\ref{tadpoleFig}(ii) and Figure~\ref{PetersenFig}(ii) also reduce to System~(\ref{C4m1Qsystem}) when
$f \in \Dl$, but they are not $\Dodd$-invariant.
\end{example}

\section{Exclusion of Other Invariant Subspaces}

Now we present a complete characterization of invariant subspaces for systems
in $\D$, $\Dzero$,  $\Dodd$ and $\Dl$.
We show that the only subspaces that are invariant for all $f$ in one of these classes
have been described by our Theorems \ref{balanced}, \ref{exo-balanced}, \ref{odd-balancedThm}, and \ref{gUlThm}.
These four theorems hypothesize a partition, or matched partition, of the cells of a graph network, 
and do not exclude the existence of invariant subspaces that do not come from a partition. 
Note that invariant subspaces of linear operators rarely come from partitions.

The next proposition characterizes all $\Dl$-invariant subspaces of $V^n$.
The significance of this is that $\Dl$ is a subset of $\Dodd,\Dzero$, and $\D$, 
so the proposition applies to all our sets of vector fields.

\begin{prop}
\label{onlyIfOdd}
Let $(G,V)$ be a graph network system. If $W$ is a $\Dl$-invariant subspace of $V^n$, then $W = \DeltaA$ for some  partition $\mathcal A$, or $W = \DeltaAm$ for some matched partition $(\mathcal A, m)$.
\end{prop}

\begin{proof}
Assume $W$ is $\Dl$-invariant. We identify $x\in V^n$ with 
$$
y=\big((x_1)_1,\ldots, (x_1)_k,\ldots,(x_n)_1,\ldots, (x_n)_k \big) \in \R^{kn}.
$$
We use Greek letters for the components of $y$ or $f(y)$.
Since $W$ is a subspace, it is the null space of a matrix with $kn$ columns in reduced row echelon form. 
Let $L$ be the index set of the leading columns and $F$ be the index set of the free columns of this matrix. 
Since the leading variables are linear combinations of the free variables, there exist real numbers $a_{\lambda, \varphi}$ such that $W$ is the set of all $y$ that satisfy
\begin{equation}
\label{Weqns}
y_\lambda = \sum_{\varphi \in F} a_{\lambda,\varphi} y_\varphi
\end{equation}
for all $\lambda \in L$.
Note that $|F| = \dim (W)$, and $|L| = kn - \dim(W)$ is the codimension of $W$.

Let $f\in \Dl$ be defined by $f_\mu(y) := y_\mu^3$, that is, $g(v_1,\ldots,v_k):=(v_1^3,\ldots,v_k^3)$ and $h=0$. 
The invariance of $W$ under $f$ implies that if $y\in W$, then
$$
(\sum_{\varphi \in F} a_{\lambda,\varphi} y_\varphi)^3=y_\lambda^3 = \sum_{\varphi \in F} a_{\lambda,\varphi} y_\varphi^3
$$
for all $\lambda \in L$.  This gives 
$$
\sum_{\varphi \in F}  (a_{\lambda, \varphi}^3 - a_{\lambda,\varphi} ) \, y_\varphi^3+ 3 \sum_{\substack{\varphi  \neq \psi \\ \varphi,\psi\in F} } a_{\lambda,\varphi}^2 a_{\lambda,\psi}  \, y_{\varphi}^2 y_{\psi} + 6 \sum_{\substack{\varphi < \psi < \theta \\ \varphi,\psi, \theta \in F} } a_{\lambda,\varphi} a_{\lambda,\psi} a_{\lambda,\theta} \, y_{\varphi} y_{\psi}  y_{\theta}= 0
$$
for all $\lambda \in L$ and $y_\varphi,y_\psi, y_\theta \in\mathbb{R}$.
Since free monomials are linearly independent, this implies
$$
a_{\lambda, \varphi}^3 - a_{\lambda,\varphi} = 0 \mbox{ and } a_{\lambda,\varphi}^2 a_{\lambda,\psi} = 0
$$
for all $\lambda \in L$ and all distinct $\varphi, \psi \in F$.
The first equation implies that each $a_{\lambda, \varphi}$ is $0$, $1$, or $-1$. 
The second equation implies that for each $\lambda \in L$ there is at most one $\varphi \in F$ for which $a_{\lambda, \varphi} \neq 0$.
Thus,  for each $\lam \in L$, Equation~(\ref{Weqns}) becomes
$y_\lambda = 0$, $y_\lambda = y_\varphi$, or $y_\lambda = - y_\varphi$
for some $\varphi \in F$.

In terms of the $x$ coordinates, this equation is $(x_i)_\ell = (x_j)_m$, $(x_i)_\ell = -(x_j)_m$, or $(x_i)_\ell = 0$,
where $(x_i)_\ell$ is a leading variable and $(x_j)_m$ is a free variable.
We now show that the equations that determine $W$ have the form $x_i = \pm x_j$ or $x_i = 0$.
The first two types of defining equations of the form $(x_i)_\ell = \pm (x_j)_m$ are impossible with $\ell \neq m$ 
since the invariance of $W$ under $f\in \Dl$ defined by 
$g(v) = (v_1, 2 v_2, \cdots , k v_k)$ and $h = 0$, namely
$(f_{i}(x))_\ell = \ell (x_i)_\ell$ implies
\begin{equation}
\label{ellemeqn}
\pm \ell (x_j)_m = \ell (x_i)_\ell  = (f_i(x))_\ell = \pm (f_j(x))_m =  \pm m (x_j)_m
\end{equation}
for all $x\in W$. 
Since $(x_j)_m$ is a free variable, 
we must have $\ell = m$.

Next, suppose that a defining equation for $W$ is $(x_i)_\ell = \pm (x_j)_\ell$ for some $\ell$.  
Then, there must be equations $(x_i)_m = \pm (x_j)_m$ for all $m \in \{1, 2, \ldots, k\}$ 
because $W$ is invariant under the coordinate rotation vector field $f\in\Dl$ defined by
$g(v_1, v_2, \ldots, v_k) = (v_2, v_3, \ldots, v_k, v_1)$ and $h = 0$, namely
\begin{equation*}
(f_i(x))_m =
\begin{cases}
 (x_i)_{m+1}, & 1\le \ell < k \\
 (x_i)_1, & m=k.
\end{cases}
\end{equation*}
Thus, the existence of a defining equation for $W$ of the form $(x_i)_\ell = \pm (x_j)_m$ implies that there are in fact $k$ equations that can be combined to $x_i = \pm x_j$.

Next, assume $W$ has a defining equation $(x_i)_\ell = 0$ for some $\ell$.  The invariance of $W$ under the coordinate rotation vector field implies that
repeated application gives
the defining equation 
$x_i = 0$.

We conclude that the defining equations for $W$ each have the form $x_i = x_j$, $x_i = - x_j$, or $x_i = 0$.

Case 1: 
If all of the equations have the form $x_i = x_j$, then $W = \DeltaA$ for some partition $\calA$.  

Case 2:
Now assume that the defining equations for $W$ include at least one equation of the form $x_i = 0$ or $x_i = - x_j$.  

We define a relation $\backsim$ on $\mathcal C$ such that $i\backsim j$ if any of the following defining equations exist:
\begin{enumerate}
\item $x_i=x_j$; 
\item $x_i=-x_k$ and $x_j=-x_k$ for some $k$;
\item $x_i=0$ and $x_j=0$.
\end{enumerate}
This relation generates an equivalence relation $\bowtie$ on $\mathcal C$, which is the intersection of all equivalence relations on $\mathcal C$ containing $\backsim$.
This equivalence relation $\bowtie$ defines a partition $\mathcal A$ of $\mathcal C$. If there is a defining equation $x_i=0$, then we define $A_0:=[i]\in\calA$. 

Let $m$ be the partial matching function $m$ defined by $m([i]) = [j]$ and $m([j]) = [i]$ if there is a defining equation $x_{i'} = - x_{j'}$ for some $i'\in [i]$ and $j'\in [j]$, and $m([i]) = [i]$ if there is a defining equation $x_{i'} = 0$ for some $i'\in [i]$. Note that $\dom(m)\not=\emptyset$ since we are in Case 2. Also, if $A_0$ exists, then $m(A_0)=A_0$.

Assume by way of contradiction that $m$ is not defined on all of $\calA$. 
Then there is a $B\in\calA\setminus\dom(m)$ and an $A\in\dom(m)$. Since $G$ is connected, we can assume that $B$ has a cell $j$ connected to a cell $i$ in $A$, so that $d_B(i)>0$.
Since $B\not\in\dom(m)$, there is an $x\in W$ such that $x_B \neq 0$ and $x_k= 0 $ for all $k \not \in B$.
Since $W$ is $\Dl$-invariant, $W$ is invariant under $f \in \Dl$ defined by $g=0$ and the identity function $h$. 
There are two subcases:

Case 2a: First, assume $A\not=A_0$. Let $i'\in m(A)$. Since $x\in W$, $x_i+x_{i'}=0$. Hence $f_i(x)+f_{i'}(x)=0$ by the invariance of $W$.
Since $i,i'\not\in B$, we must have $x_{[i]}=0=x_{[i']}$. So Equation~(\ref{giProof}) becomes
$$
f_i(x) = d_B(i) x_B, \quad
f_{i'}(x) = d_B(i') x_B.
$$
Since $d_B(i) > 0$ and $d_B(i') \geq 0$, we have $f_i(x) + f_{i'}(x) = (d_B(i)+d_B(i'))x_B\neq 0$, which is a contradiction.

Case 2b: Next, assume $A=A_0$.
Since $x\in W$, $x_i=0$. Hence $f_i(x)=0$ by the invariance of $W$.
So Equation~(\ref{giProof}) becomes
$$
f_i(x) = d_B(i) x_B \neq 0,
$$
which is again a contradiction.

Therefore $\dom(m)=\calA$. 
If $A_0 \in \calA$, then $\calA$ is an odd partition.
If $A_0 \not \in \calA$, then we add $A_0 := \emptyset$ to $\calA$ to get an odd partition, and we extend $m$ with $m(A_0)=A_0$. 
In both cases $(\calA,m)$ is a matched partition, and $W = \DeltaAm$.
\end{proof}

Our main theorem then follows from this proposition and the previous theorems.
Figure~\ref{mainTheoremFig} summarizes this result with a diagram.

\begin{thm}
\label{MainTheorem}
If $G$ is a graph network and $W$ is a subspace of the phase space $V^n$, then the following hold.
\begin{enumerate}
\item[$(1)$] $W$ is $\Dl$-invariant if and only if $W$ is linear-balanced or exo-balanced.
\item[$(2)$] $W$ is $\Dodd$-invariant if and only $W$ is odd-balanced or exo-balanced.
\item[$(3)$] $W$ is $\Dzero$-invariant if and only if $W$ is exo-balanced.
\item[$(4)$] $W$ is $\D$-invariant if and only if $W$ is balanced.
\item[$(5)$] $W$ is $\mathcal F_G$-invariant if and only if $W$ is balanced.
\end{enumerate}
\end{thm}

\begin{proof}
Note that $\Dl$ is a subset of $\Dodd,\Dzero$, $\D$, and $\mathcal F_G$. Hence Proposition~\ref{onlyIfOdd} applies in the forward direction of every case.

(1)
Assume $W$ is a $\Dl$-invariant subspace of $V^n$.
Proposition~\ref{onlyIfOdd} implies that 
$W = \DeltaAm$ for some matched partition $\calAm$, or $W = \DeltaA$ for some partition $\calA$. 
If $W = \DeltaAm$, then $W$ is linear-balanced by Theorem \ref{gUlThm}. 
If $W= \DeltaA$, then $W$ is exo-balanced by Theorem~\ref{exo-balanced}.

Conversely, if $W$ is linear-balanced, then $W$ is $\Dl$-invariant by Theorem~\ref{gUlThm}.
If $W$ is exo-balanced, then $W$ is $\Dl$-invariant by Theorem~\ref{exo-balanced}.

(2) Assume $W$ is a $\Dodd$-invariant subspace of $V^n$. 
Proposition~\ref{onlyIfOdd} implies that $W = \DeltaAm$, or $W = \DeltaA$. 
If $W = \DeltaAm$, then $W$ is odd-balanced by Theorem \ref{odd-balancedThm}.  
If $W= \DeltaA$, then $W$ is exo-balanced by Theorem~\ref{exo-balanced}.

Conversely, if $W$ is odd-balanced, then $W$ is $\Dodd$-invariant by Theorem~\ref{odd-balancedThm}.
If $W$ is exo-balanced, then $W$ is $\Dodd$-invariant by Theorem~\ref{exo-balanced}.

(3) Assume $W$ is a $\Dzero$-invariant subspace of $V^n$. Proposition \ref{onlyIfOdd} implies that $W = \DeltaAm$ or $W = \DeltaA$.  
However $W = \DeltaAm$ is not possible since $\DeltaAm$ is not invariant under $f \in \Dzero$
defined by $f_i(x) =  v$, for any nonzero constant $v \in V$.  
Thus $W = \DeltaA$ is exo-balanced by Theorem~\ref{exo-balanced}.
Conversely, if $W$ is is exo-balanced, then $W$ is $\Dzero$-invariant by Theorem~\ref{exo-balanced}.

(4) Assume $W$ is a $\D$-invariant subspace of $V^n$. Proposition \ref{onlyIfOdd} implies that $W = \DeltaAm$ or $W = \DeltaA$.  
By the argument used in Case (3) above,  $W = \DeltaAm$ is not possible.
Thus $W = \DeltaA$ is balanced by Theorem~\ref{balanced}.
Conversely, if $W$ is is balanced, then $W$ is $\D$-invariant by Theorem~\ref{balanced}.

(5) Assume $W$ is an $\mathcal F_G$-invariant subspace of $V^n$. 
As in Cases (3) and (4), Proposition~\ref{onlyIfOdd} and the the argument used in Case (3) show that
$W = \DeltaA$.
Thus $W = \DeltaA$ is balanced by 
Theorem~\ref{balanced} or \cite[Theorem 6.5]{Pivato}.
Conversely, if $W$ is is balanced, then $W$ is $\mathcal F_G$-invariant by \cite[Theorem 6.5]{Pivato}.
\end{proof}

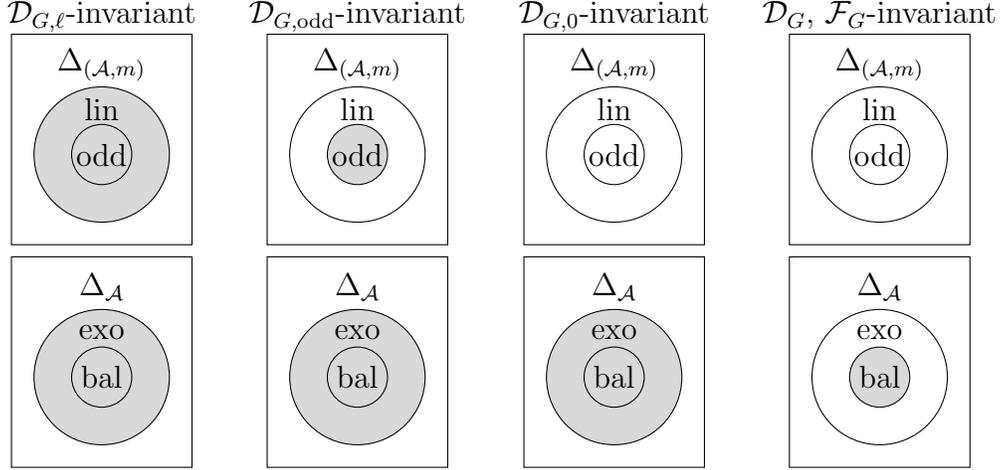
\begin{figure}
\begin{tabular}{ccccccc}
$\mathcal{D}_{G,\ell}$-invariant &  & $\mathcal{D}_{G,\text{odd}}$-invariant &  & $\mathcal{D}_{G,0}$-invariant &  & $\mathcal{D}_{G}$, $\mathcal{F}_{G}$-invariant\tabularnewline
\begin{tikzpicture}
\draw (-1.2,-1.2) rectangle (1.2,1.6);
\draw[fill=black!15] (0,0) ellipse (.9cm and .9cm);
\draw (0,0) circle (.4cm);
\node at (0,0) {odd};
\node at (0,.6) {lin};
\node at (0,1.2) {$\Delta_{(\mathcal{A},m)}$};
\end{tikzpicture} &  & \begin{tikzpicture}
\draw (-1.2,-1.2) rectangle (1.2,1.6);
\draw (0,0) ellipse (.9cm and .9cm);
\draw[fill=black!15] (0,0) circle (.4cm);
\node at (0,0) {odd};
\node at (0,.6) {lin};
\node at (0,1.2) {$\Delta_{(\mathcal{A},m)}$};
\end{tikzpicture} &  & \begin{tikzpicture}
\draw (-1.2,-1.2) rectangle (1.2,1.6);
\draw (0,0) ellipse (.9cm and .9cm);
\draw (0,0) circle (.4cm);
\node at (0,0) {odd};
\node at (0,.6) {lin};
\node at (0,1.2) {$\Delta_{(\mathcal{A},m)}$};
\end{tikzpicture} &  & \begin{tikzpicture}
\draw (-1.2,-1.2) rectangle (1.2,1.6);
\draw (0,0) ellipse (.9cm and .9cm);
\draw (0,0) circle (.4cm);
\node at (0,0) {odd};
\node at (0,.6) {lin};
\node at (0,1.2) {$\Delta_{(\mathcal{A},m)}$};
\end{tikzpicture}\tabularnewline
\begin{tikzpicture}
\draw (-1.2,-1.2) rectangle (1.2,1.6);
\draw[fill=black!15] (0,0) ellipse (.9cm and .9cm);
\draw (0,0) circle (.4cm);
\node at (0,0) {bal};
\node at (0,.6) {exo};
\node at (0,1.2) {$\Delta_\mathcal{A}$};
\end{tikzpicture} &  & \begin{tikzpicture}
\draw (-1.2,-1.2) rectangle (1.2,1.6);
\draw[fill=black!15] (0,0) ellipse (.9cm and .9cm);
\draw (0,0) circle (.4cm);
\node at (0,0) {bal};
\node at (0,.6) {exo};
\node at (0,1.2) {$\Delta_\mathcal{A}$};
\end{tikzpicture} &  & \begin{tikzpicture}
\draw (-1.2,-1.2) rectangle (1.2,1.6);
\draw[fill=black!15] (0,0) ellipse (.9cm and .9cm);
\draw (0,0) circle (.4cm);
\node at (0,0) {bal};
\node at (0,.6) {exo};
\node at (0,1.2) {$\Delta_\mathcal{A}$};
\end{tikzpicture} &  & \begin{tikzpicture}
\draw (-1.2,-1.2) rectangle (1.2,1.6);
\draw (0,0) ellipse (.9cm and .9cm);
\draw[fill=black!15] (0,0) circle (.4cm);
\node at (0,0) {bal};
\node at (0,.6) {exo};
\node at (0,1.2) {$\Delta_\mathcal{A}$};
\end{tikzpicture}\tabularnewline
\end{tabular}

\caption{
Visualization of Theorem~\ref{MainTheorem} for a fixed graph network system $(G,V)$.
The rectangles represent the polydiagonal subspaces $\DeltaAm$ or $\DeltaA$.
The circles represent linear-balanced, odd-balanced, exo-balanced, or balanced subspaces.
The subspaces that are invariant under the indicated set of vector fields are shaded.
}
\label{mainTheoremFig}
\end{figure}

\begin{conj}
If $\calAm$ is a linear-balanced partition of the cells of a graph network, then $|A| = |-A|$ for all $A\in\calA$.
\end{conj}

We have a proof of our conjecture for the weaker odd-balanced case, which we present after the following lemmas.

\begin{lem}
\label{OddSizes}
Assume $\calAm$ is an odd-balanced partition of the cells of a graph network, and $A, B \in \calA$ satisfy $A \ne B$ and 
$A \ne A_0 \ne B$.  Then
$$
|A| \; d_B(A) = |B| \; d_A(B).
$$
\end{lem}

\begin{proof}
The proof is the same as that of Proposition~\ref{ABedges}. 
The hypotheses ensure that $d_A(B)$ and $d_B(A)$ are defined.
Note that $B = -A$ is allowed.
\end{proof}

\begin{lem}
\label{joinedToA0}
Assume $\calAm$ is an odd-balanced partition of the cells of a graph network, 
and $A \in \calA$ satisfies $A \neq A_0$ and $d_{A_0}(A) \ne 0$. Then $|A| = |-A|$.
\end{lem}

\begin{proof}
Condition (1) of Definition \ref{odd-balancedDef} implies that $d_{A_0}(A) = d_{A_0}(-A)$.
Condition (2) says that $d_A(i) = d_{-A}(i)$ for each $i \in A_0$.
The total number of edges between cells in $A_0$ and cells in $A$ is the sum of
$d_A(i)$ over $i \in A_0$. Dividing by the degree $d_{A_0}(A)$ we find
$$
|A| = \frac{1}{d_{A_0}(A) } \sum_{i \in A_0} d_A(i) = 
\frac{1}{d_{A_0}(-A) } \sum_{i \in A_0} d_{-A}(i) = |-A|.
$$
\end{proof}

\begin{prop}
If $\calAm$ is an odd-balanced partition of the cells of a graph network, then $|A| = |-A|$ for all $A\in\calA$.
\end{prop}

\begin{proof}
Let $i\in A$ and $i'\in-A$. Since the graph network is connected, we can
find a path from $i$ to $i'$. Let $A_{1},A_{2},\ldots,A_{n}$ be
the sequence of equivalence classes containing the cells of this
path, so that $i\in A_{1}$, $i'\in A_{n},$ and $A_{j}\ne A_{j+1}$
for all $j$. Note that $A_{j}$ can contain several consecutive terms
of the path. Then $d_{A_{j}}(A_{j+1})\ne0$ for all $j$. 

Case 1: Assume $A_{k}\ne A_{0}$ for all $k$. Using the Lemma~\ref{OddSizes} we
have 
\[
|A|=|A_{1}|=\frac{d_{A_{1}}(A_{2})}{d_{A_{2}}(A_{1})}|A_{2}|=\cdots=\frac{d_{A_{1}}(A_{2})}{d_{A_{2}}(A_{1})}\cdot\frac{d_{A_{2}}(A_{3})}{d_{A_{3}}(A_{2})}\cdots\frac{d_{A_{n-1}}(A_{n})}{d_{A_{n}}(A_{n-1})}|A_{n}|.
\]
We also have 
\[
|-A|=|-A_{1}|=\frac{d_{-A_{1}}(-A_{2})}{d_{-A_{2}}(-A_{1})}|-A_{2}|=\cdots=\frac{d_{-A_{1}}(-A_{2})}{d_{-A_{2}}(-A_{2})}\cdot\cdots\frac{d_{-A_{n-1}}(-A_{n})}{d_{-A_{n}}(-A_{n-1})}|-A_{n}|.
\]
Since $d_{A_{j}}(A_{j+1})=d_{-A_{j}}(-A_{j+1})$ for all $j$, we
have 
\[
\frac{|A|}{|-A|}=\frac{|A|}{|A_{n}|}=\frac{|-A|}{|-A_{n}|}=\frac{|-A|}{|A|}.
\]
Hence $|A|^{2}=|-A|^{2}$ which implies $|A|=|-A|$.

Case 2: Assume $A_{k+1}=A_{0}$ for some $k$. We can also assume
that $A_{j}\ne A_{0}$ for $j\le k$. Now we have
\[
|A|=|A_{1}|=\frac{d_{A_{1}}(A_{2})}{d_{A_{2}}(A_{1})}|A_{2}|=\cdots=\frac{d_{A_{1}}(A_{2})}{d_{A_{2}}(A_{1})}\cdot\frac{d_{A_{2}}(A_{3})}{d_{A_{3}}(A_{2})}\cdots\frac{d_{A_{k-1}}(A_{k})}{d_{A_{k}}(A_{k-1})}|A_{k}|.
\]
We also have 
\[
|-A|=|-A_{1}|=\frac{d_{-A_{1}}(-A_{2})}{d_{-A_{2}}(-A_{1})}|-A_{2}|=\cdots=\frac{d_{-A_{1}}(-A_{2})}{d_{-A_{2}}(-A_{2})}\cdot\cdots\frac{d_{-A_{k-1}}(-A_{k})}{d_{-A_{k}}(-A_{k-1})}|-A_{k}|.
\]
Since $d_{A_{j}}(A_{j+1})=d_{-A_{j}}(-A_{j+1})$ for all $j$, we
have

\[
\frac{|A|}{|A_{k}|}=\frac{|-A|}{|-A_{k}|}.
\]
By Lemma~\ref{joinedToA0}, $|A_{k}|=|-A_{k}|$. Thus $|A|=|-A|$.
\end{proof}

\section{The Lattice of Invariant Subspaces}

For a given graph network system $(G, V)$, the $\Dl$-invariant subspaces of the four types (balanced, exo-balanced, odd-balanced and linear-balanced) ordered by reversed inclusion
form a partially ordered set (poset). The join of a set $S$ of $\Dl$-invariant subspaces is $\bigvee(S)=\bigcap S$. The maximum and minimum elements are $\{(0, 0, 0, \ldots ) \}$ and $\{(a, b, c, \ldots )\}$ respectively.
Hence this poset is actually a complete lattice by \cite[Theorem 2.31]{Davey}.
The lattice of balanced subspaces on more general cell networks is studied in \cite{StewartLattice}.

We visualize this lattice by drawing the Hasse diagram of the \emph{orbit quotient poset}. 
The elements of this quotient poset are the orbits of the $\aut(G)$ action on the set of invariant subspaces. The partial order is defined by $[U]\le [W]$ if there is a $\phi\in \aut(G)$ such that $\phi\cdot W\subseteq U$.
Note that the quotient poset might not be a lattice. For example, our computer calculations show that the 37-element orbit quotient poset for the 8-vertex cube graph is not a lattice.
Before identification there are 142 $\Dl$-invariant subspaces that 
form a lattice, as required by the general theory.

This calculation was done using a brute-force algorithm that we implemented in C++. 
The algorithm computes the lattice for networks with fewer than about 15 
cells in a reasonable amount of time. The algorithm in \cite{KameiLattice} is similar to ours, except it uses a matrix description of the balanced condition. 
Our algorithm simply generates all partitions and matched partitions and checks the appropriate conditions.
A more efficient algorithm, using eigenvectors of the Laplacian matrix, could possibly be constructed using Remarks~\ref{LapExo-balanced} and \ref{Linvariant}, along the lines of \cite{Aguiar&Dias}.

A computation of the lattice of $\Dl$-invariant subspaces for a cell network $G$ is the first step toward
understanding the bifurcations that occur in systems of differential
equations with vector fields in $\Dl$.  These bifurcations are currently only partially understood.

Several examples of lattices of $\Dl$-invariant subspaces are shown in Figures~ \ref{tadpoleLattice},
\ref{nosym4and7Hasse},
 \ref{P6Hasse}, and \ref{P3C3Hasse}. 
In these figures, 
the matched subspaces $\DeltaAm$ are shown with a double border, and the un-matched subspaces $\DeltaA$ have a single border. 
Strictly exo-balanced subspaces
and strictly linear-balanced subspaces are shown with a shaded background, and the others have a white background.

In our Hasse diagram figures, subspaces at the same height have the same dimension.
The solid arrows connect subspaces with a different point stabilizer within $\aut(G) \times \Z_2$, and the
dashed arrows connect subspaces with the same stabilizer.  Hence, the
invariant subspace is a fixed point subspace precisely when there are no dashed arrows leaving
the subspace.
Thus, dashed arrows are never symmetry-breaking bifurcations described by standard equivariant bifurcation theory.
That theory describes solid arrows connecting two fixed-point subspaces.  
However, in many cases there is a standard symmetry-breaking bifurcation
within the reduced system for the invariant subspace of a daughter branch.

\begin{figure}[h]
\includegraphics[scale = 1.2]{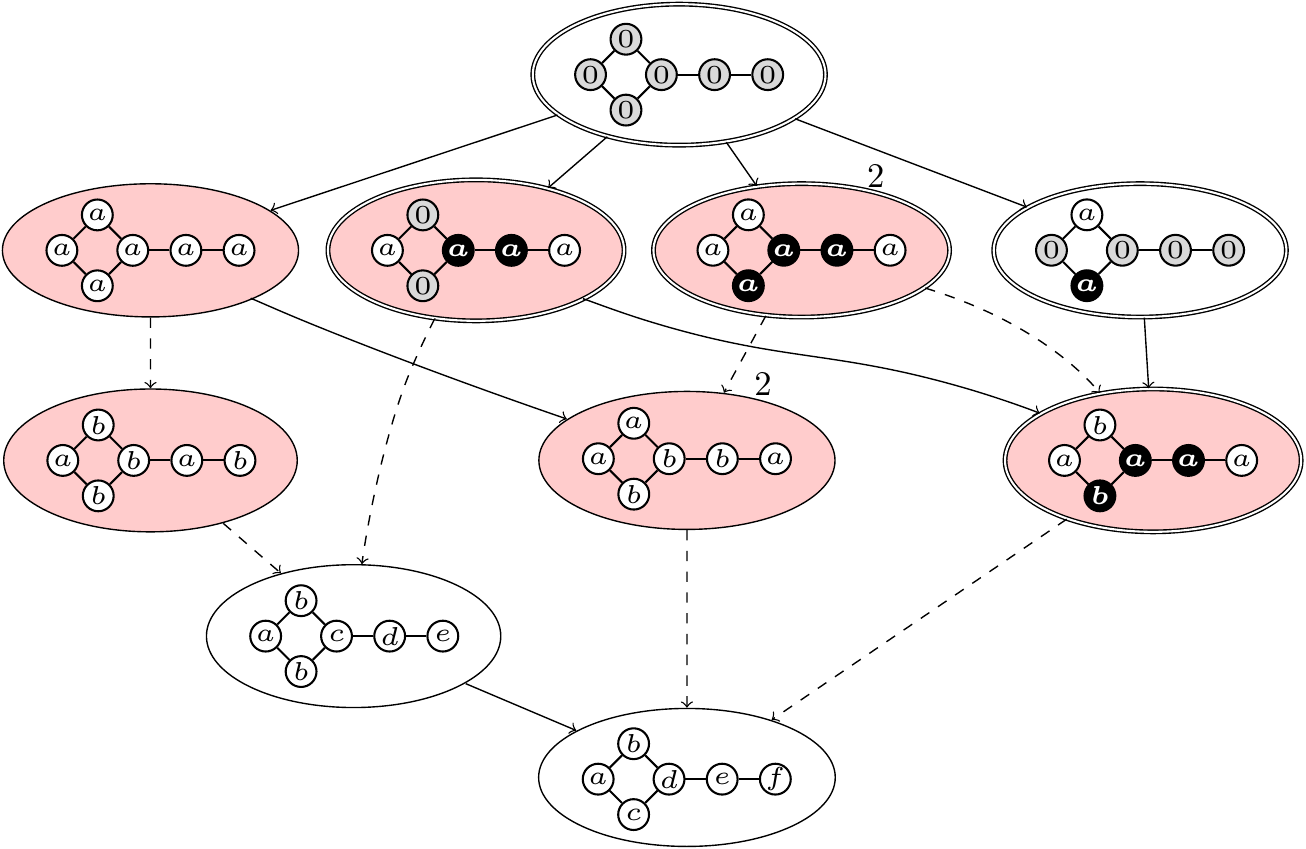}
\caption{
The lattice of $\Dl$-invariant subspaces on a 6-cell graph network. Each $\aut(G)$ orbit is shown with a representative.
The numeral ``2'' indicates that the $\aut(G)$ orbit of the subspace has 2 elements.
The balanced subspaces have a single border and white background.  The exo-balanced are single border, shaded.  The odd balanced are double border, white,
and the linear-balanced are double border, shaded. 
}
\label{tadpoleLattice}
\end{figure}


\begin{example}
The lattice of $\Dl$-invariant subspaces for a certain network with all four types of invariant subspaces is shown in Figure~\ref{tadpoleLattice}. 
This figure can be used to obtain the lattice for any of the classes of vector fields.  
For this network, the three restricted cases are as follows:
The lattice of $\D$-invariant subspaces is the sublattice composed of the two balanced subspaces,
denoted by the single border and white background.
The  lattice of $\Dzero$-invariant subspaces
is the sub-lattice composed of the 5 exo-balanced subspaces with a single border.
The lattice of $\Dodd$-invariant subspaces is the sub-lattice with the 7 subspaces excluding the 3 strictly linear-balanced subspaces,
indicated by shading and double-borders. 
\end{example}

\begin{figure}[h]
\begin{tabular}{ccc}
\includegraphics[scale=1.2]{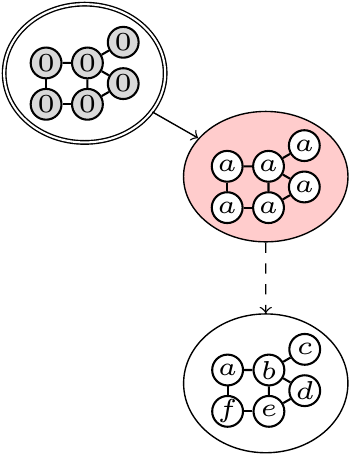}
   & $\qquad\qquad$ &
\includegraphics[scale=1.2]{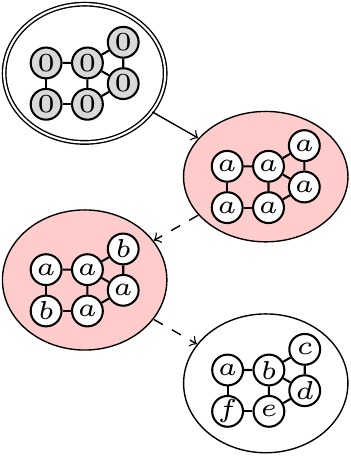}
  \\
(i) & & (ii) 
\end{tabular}
\caption{The lattice of $\Dl$-invariant subspaces for two different coupled cell networks, each with trivial $\aut(G)$. 
}\label{nosym4and7Hasse}
\end{figure}


\begin{example}
For every network $G$ with more than 1 cell, the lattice of $\Dl$-invariant subspaces has at least three subspaces: 
 $\{(0, 0, 0, \ldots) \}$,
 $\{(a, a, a, \ldots)\}$, and
 $\{(a, b, c, \ldots)\}$.  
If $\aut(G)$ is nontrivial, then there will certainly be more invariant subspaces. 
On the other hand, if $\aut(G)$ is trivial, then there may or may not be more invariant subspaces.
Figure~\ref{nosym4and7Hasse} shows two examples of networks with trivial $\aut(G)$.  Figure~\ref{nosym4and7Hasse}(i) shows
a lattice with the minimum number of $\Dl$-invariant subspaces, 
and Figure~\ref{nosym4and7Hasse}(ii) shows the lattice for a different network, with one more edge.  
This second lattice has one extra exo-balanced subspace.
\end{example}


\begin{figure}
\includegraphics[scale=1.2]{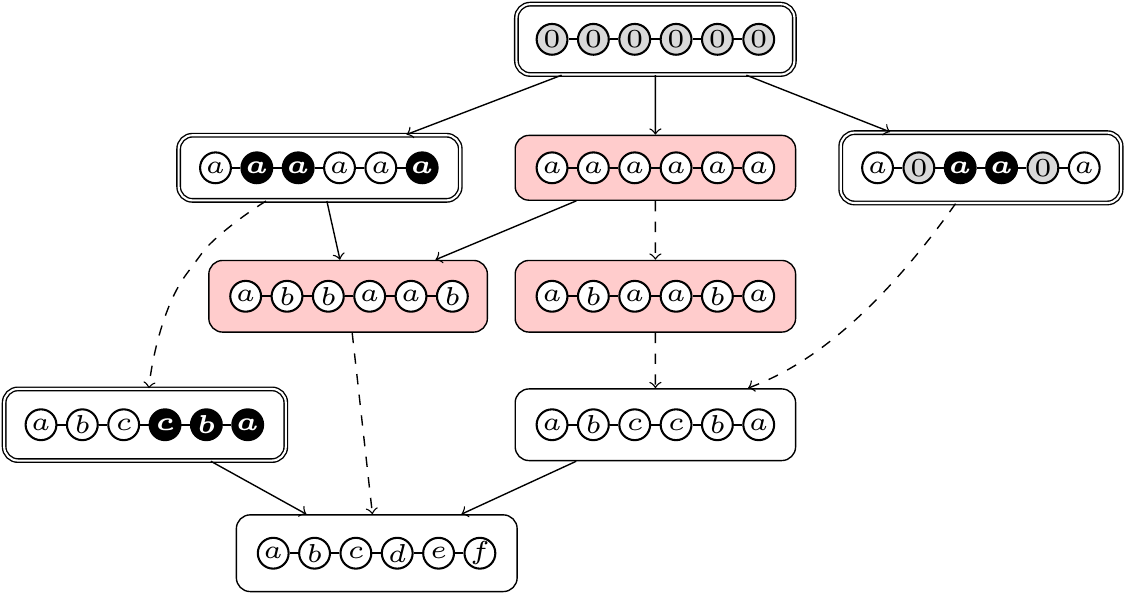}
\caption{
The lattice of $\Dl$-invariant subspaces for $G = P_6$. 
}
\label{P6Hasse}
\end{figure}

\begin{example}
The network of $n$ coupled cells in a path $G=P_n$ can be analyzed fairly completely.  See for example  \cite{EpsteinGolubitsky}. 
The lattice of $\Dl$-invariant subspaces for $n = 6$ is shown in Figure~\ref{P6Hasse}.  
This lattice was found with an exhaustive search of all partitions. This example suggests the following algorithm to construct all $\Dl$-invariant subspaces for $G$.


We start with three types of basic partitions of the cells of $P_q$. The generic basic partitions are   
\[
\mathcal{G}_1:=(a),\ \mathcal{G}_2:=(a,b),\ \mathcal{G}_3:=(a,b,c),\ \ldots
\] 
The even basic partitions are defined for odd $q$ satisfying $q\geq 3$. They are 
\[
\mathcal{E}_3:=(a,b,a),\ \mathcal{E}_5:=(a,b,c,b,a),\ \mathcal{E}_7:=(a,b,c,d,c,b,a),\ \ldots
\]  
The odd basic partitions are the matched partitions
\[
\mathcal{O}_1:=(0),\ \mathcal{O}_2:=(a,-a),\ \mathcal{O}_3:=(a,0,-a),\ \mathcal{O}_4:=(a,b,-b,-a),\ \ldots
\]

Stringing together $k$ copies of a basic partition $\mathcal{S}$, and the reverse $\mathcal{S}^r$ of this basic partition, we get a partition of the cells of $P_{kr}$ of the form $\mathcal{A}=\mathcal{S}\mathcal{S}^r \mathcal{S} \cdots \mathcal{S}^r  \mathcal{S}$ for $k$ odd, and $\mathcal{A}=\mathcal{S} \mathcal{S}^r \mathcal{S} \cdots  \mathcal{S} \mathcal{S}^r$ for $k$ even. It is easy to see 
that every basic partition produces different partitions of the cells of $P_n$. For a given $n$, the partitions of the cells of $P_n$ are obtained by using all of the factorizations $n = q k$. The resulting partition $\mathcal{A}$ satisfies the following:
\begin{enumerate}
\item
If $\mathcal{S} = \mathcal{G}_q$, then $\mathcal{A}$ is balanced for $k = 1$ or $2$, and strictly exo-balanced for $k > 2$.
\item
If $S =\mathcal{E}_q$, then $\mathcal{A}$ is balanced for $k = 1$, and strictly exo-balanced for $k > 1$. 
\item
If $S = \mathcal{O}_q$, then $\mathcal{A}$ is odd-balanced.
\end{enumerate}

The partitions of Figure~\ref{P6Hasse} are all found by our algorithm. The list below shows the basic partitions together 
with their corresponding partitions of the cells of $P_6$: 
\vspace{-4mm}
\begin{multicols}{3}
$\mathcal{G}_1$: $(a,a,a,a,a,a)$

$\mathcal{O}_1$: $(0,0,0,0,0)$

$\mathcal{G}_2$: $(a,b,b,a,a,b)$

$\mathcal{O}_2$: $(a, -a, -a, a, a, -a)$

$\mathcal{G}_3$: $(a,b,c,c,b,a)$ 

$\mathcal{E}_3$: $(a, b, a, a, b, a)$ 

$\mathcal{O}_3$: $(a, 0, -a, -a, 0, a)$ 

$\mathcal{G}_6$: $(a, b,c, d, e, f)$

$\mathcal{O}_6$: $(a, b, c, -c, -b, -a)$
\end{multicols}
\end{example}

\begin{conj}
\label{pathConj}
We conjecture that the algorithm in the previous example gives all of the $\Dl$-invariant subspaces for the path $G=P_n$.
\end{conj}

\begin{figure}
\includegraphics[scale=1.2]{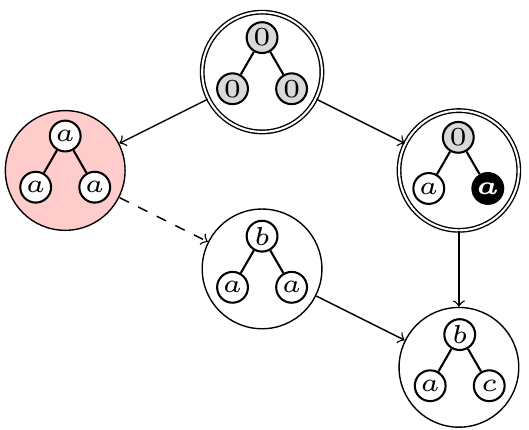}
\quad
\includegraphics[scale=1.2]{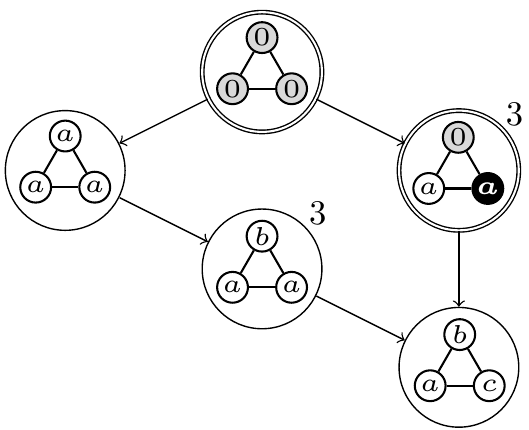}
\caption{The lattice of $\Dl$-invariant subspaces for the networks $P_3$ (the path) and $C_3$ (the cycle).
The numeral ``3'' indicates that the $\aut(G)$ orbit of the subspace has 3 elements.
}
\label{P3C3Hasse}
\end{figure}

\begin{example}
Figure~\ref{P3C3Hasse} shows the lattice of invariant subspaces for the two connected networks with 3 cells. 
Note that all $\Dl$-invariant subspaces are fixed point subspaces for $C_3$. 
On the other hand, $P_3$ has a strictly exo-balanced partition (shaded on the figure) that is not a fixed point subspace. 
We conjecture that all $\Dl$-invariant subspaces are fixed point subspaces for every complete graph $G$.
\end{example}

\section{Application to Coupled van der Pol Oscillators}

This section gives several examples of coupled generalized van der Pol oscillators. We show how various choices of the system parameters
allow the vector field to be in $\D$, $\Dzero$, $\Dodd$, or $\Dl$.

Consider the difference-coupled vector field on a graph network system $(G,\mathbb{R}^2)$, where each cell is a van der Pol oscillator.  
As in Section \ref{springsEx} we use $x_i = (u_i, \dot u_i) \in \R^2$ to describe the state of each oscillator.
The simplest such system is
\begin{equation}
\label{vdPlc}
\ddot u_i = \alpha(1- u_i^2) \dot u_i - u_i + \sum_{j \in N(i)} \delta(u_j - u_i),
\end{equation}
for each $i \in \{i, 2, \ldots, n\}$, where $\alpha$ and $\delta$ are real parameters.
Oscillator $i$ is coupled to its neighbors $N(i)$ in the graph network.

To illustrate the special nature of System~(\ref{vdPlc}), we will study a 
more general system of coupled van der Pol oscillators where the equations of motion are
\begin{equation}
\label{vdP}
\ddot u_i = \alpha(1- u_i^2) \dot u_i - u_i + \beta u_i^2 +
\sum_{j \in N(i)} \big ( \gamma + \delta (u_j - u_i) + \eps (u_j - u_i)^3 \big ).
\end{equation}
This dynamical system has 5 real parameters:
$\alpha, \beta,\gamma, \delta$, and $\eps$.
Small $\alpha$ gives near-circular limit cycles of the oscillators, whereas large $\alpha$ causes
a relaxation oscillation.  
Nonzero $\beta$ makes the internal dynamics non-odd;
the classic van der Pol oscillator has $\beta = 0$.
The $\gamma$ term can describe massive coupling springs, as in Example~\ref{springsEx}.
Positive coupling constants $\delta,\eps$ pull neighboring oscillators toward the same state,
and tend to synchronize the oscillators,
whereas negative coupling constants
push the oscillators away from each other.

System (\ref{vdP}) can be written as a graph network dynamical system $\dot x = f(x)$ with phase space $ (\R^2)^n$ by defining
$g(u, v) = (v, \alpha (1-u^2)v  -u + \beta u^2)$ and
$h(u, v)  = (0, \gamma + \delta u + \eps u^3)$. 
Note that
\begin{itemize}
\item $f \in \D$;
\item $f \in \Dzero$ if and only if $\gamma = 0$;
\item $f \in \Dodd$ if and only if $\gamma = \beta = 0$;
\item $f \in \Dl$ if and only if $\gamma =\beta = \eps = 0$.
\end{itemize}
This section illustrates how the simplest models often have special properties.
For example, System~(\ref{vdPlc}) is often used as a model in a situation when System~(\ref{vdP}) would be more appropriate
because strictly linearly balanced subspaces are not invariant for the physical system.
\begin{figure}[ht]
\begin{tabular}{cc}
\includegraphics[scale = 1]{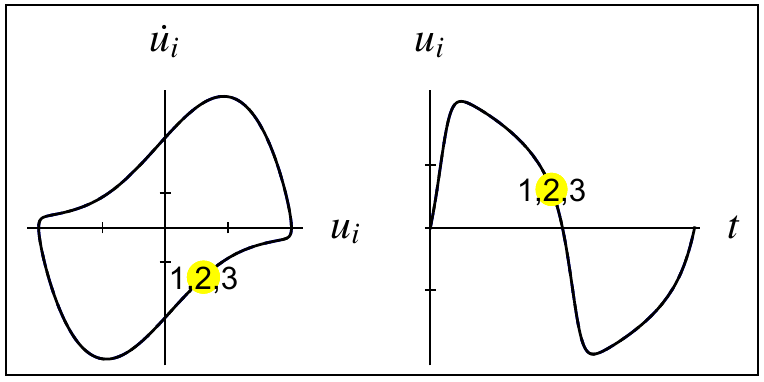} 
&
\includegraphics[scale = 1]{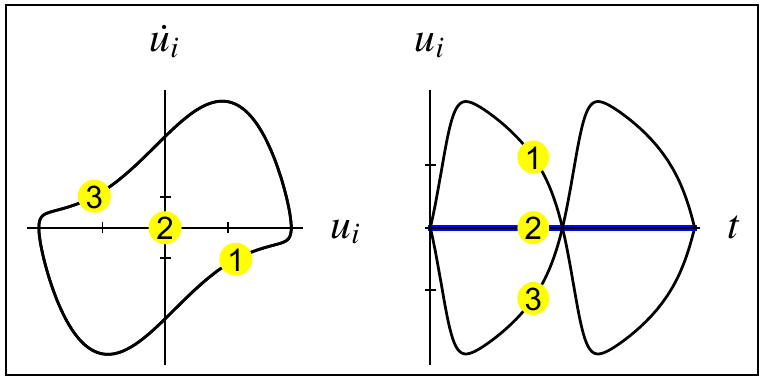}
  \\
(i) &  (ii) 
\end{tabular}
\caption{
Two solutions of the coupled van der Pol System (\ref{vdP})
on the path $P_3$, with $f \in \Dl$. 
See Example~\ref{vdP3} for the parameters.
The numbers in the yellow dots indicate $i$.
The tick marks are at $u_i = \pm 1$ and $\dot u_i = \pm 1$.
The solution in Figure (i) has $u_1 = u_2 = u_3$.  In Figure (ii), $u_1 + u_3 = u_2 = 0$.
}
\label{P3Dl}
\end{figure}
\begin{figure}[ht]
\begin{tabular}{cc}
\includegraphics[scale = 1]{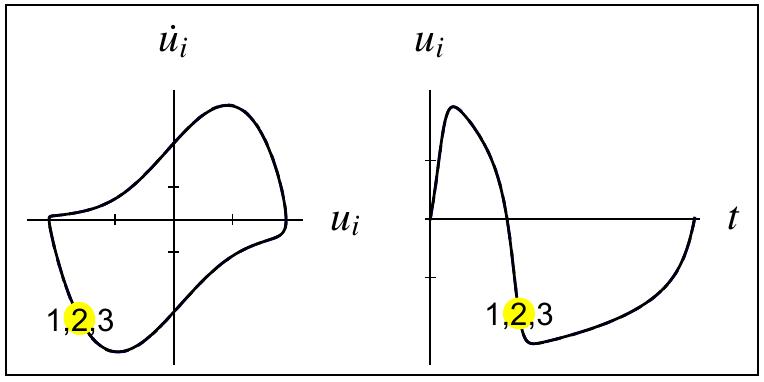}
&
\includegraphics[scale = 1]{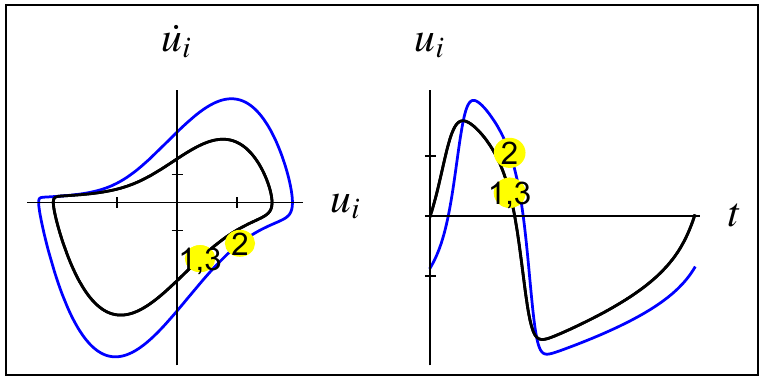} \\
(i) & (ii)
\end{tabular}
\caption{Two solutions of the coupled van der Pol System (\ref{vdP})
on the path $P_3$.  The solution in Figure~(i) has $f\in\Dzero\setminus\Dodd$, and
the solution has $u_1 = u_2 = u_3$.
In Figure~(ii), $f \in \D \setminus \Dzero$, and
the solution has $u_1 = u_3$.
See Example~\ref{vdP3}.
}
\label{P3D}
\end{figure}

\begin{example}
\label{vdP3}
Consider the system of 3 coupled van der Pol oscillators based on the graph network  $G = P_3$.
The lattice of invariant subspaces shown on the left in Figure~\ref{P3C3Hasse} can be used to understand and classify the observed solutions,
but not to predict which patterns of oscillation are stable.
Figure~\ref{P3Dl} shows two solutions to System~(\ref{vdP}) with $f \in \Dl$ (with nonzero parameters $\alpha = 2$ and $\delta = 1$).  
Figure~\ref{P3Dl}(i) shows an attracting solution in the exo-balanced subspace $\{ (a,a,a) \mid a \in \R^2\}$.  
Letting $a = (u, v)$,
the reduced System~(\ref{reducedExoSys}) is the single van der Pol equation 
$$
\ddot u = 2 (1-u^2) \dot u -u
$$
and there is no coupling between the oscillators.
Figure~\ref{P3Dl}(ii) shows a repelling solution in the linear-balanced subspace $\{ (a,0,-a) \mid a \in \R^2\}$.  The reduced system 
$$
\ddot u = 2 (1-u^2) \dot u - 2 u
$$
is slightly different from the previous case, and again there is effectively a single oscillator.

Figure~\ref{P3D} shows two attracting solutions to System~(\ref{vdP}) that are described by the reduced systems in Example~\ref{abExosystems}.
Figure~\ref{P3D}(i) has $f \in \Dzero \setminus  \Dodd$, with nonzero parameters $\alpha = 2$, $\beta =-0.3$ and $\delta = 1$.  The  solution is in the same exo-balanced subspace as the solution in Figure~\ref{P3Dl}(i), but now the reduced system is 
$$
\ddot u = 2 (1-u^2) \dot u -u - 0.3 u^2 .
$$

Figure~\ref{P3D}(ii) has $f \in \D\setminus\Dzero$, with nonzero parameters $\alpha = 2$, $\gamma = 0.5$, and $\delta = -0.5$. 
The solution is in the  balanced subspace $\{(a, b, a) \mid a, b \in \R^2\}$.
Letting $a = (u_1, v_1)$ and $b = (u_2, v_2)$, the reduced system is
\begin{align*} 
\ddot u_1 &= 2 (1-u_1^2) \dot u_1 - u_1 - 0.5 + 0.5(u_2-u_1)  \\
\ddot u_2 &= 2 (1-u_2^2) \dot u_2 - u_2 - 1 + (u_1 - u_2) .
\end{align*}
For this graph network system system $(P_3, \R^2)$, the only $\D$-invariant subspaces are the aforementioned balanced subspace with $x_1 = x_3$ and the full phase space $\R^6$.
Thus, Figure~\ref{P3D}(ii) shows the only nontrivial symmetry of solutions expected when $\gamma \neq 0$.
\end{example}

\begin{figure}
\begin{tabular}{ccc}
\includegraphics[scale = 1]{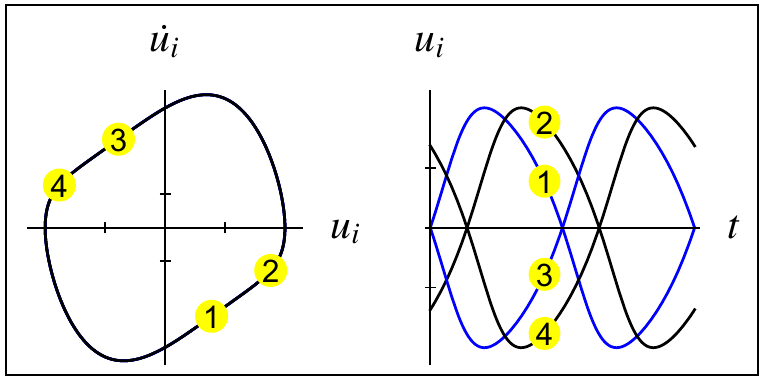}
   & 
\includegraphics[scale = 1]{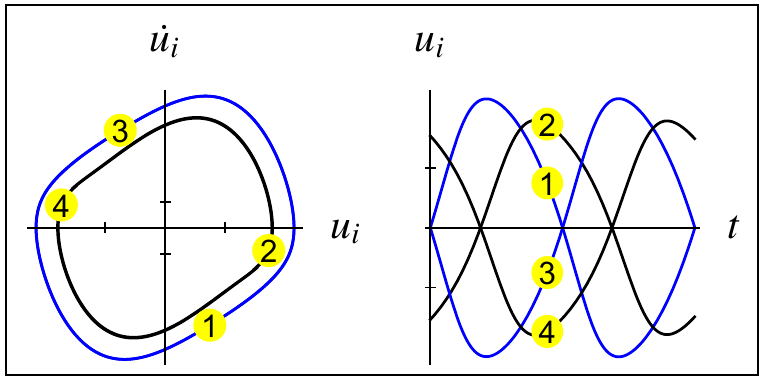}
  \\
(i) &  (ii) 
\end{tabular}
\caption{Two solutions of the coupled van der Pol system~(\ref{vdP}) on the square graph $C_4$. 
The two solutions are in the same odd-balanced subspace,
with $u_1 + u_3 = u_2+ u_4 = 0$. In (i) $f\in\Dl$ while in (ii) $f\in\Dodd\setminus\Dl$.
The solution shown in (i) has two decoupled anti-synchronized pairs of oscillators with an arbitrary phase shift between the pairs.
See Example~\ref{C4Ex}.
}
\label{C4ab-a-bFig}
\end{figure}
\begin{example}
\label{C4Ex}
There are 16 $\Dl$-invariant subspaces for the graph network $G = C_4$. 
Every invariant subspace is a fixed point subspace,
thus the lattice of invariant subspaces is the lattice of fixed point subspaces.
While the lattice of invariant subspaces does not include anything new, the reduced systems are interesting.
For example, the dynamics on the odd-balanced subspace $\{(a, b, -a, -b) \mid a, b \in V \}$ is quite different for $f \in \Dl$ and for $f \in \Dodd \setminus \Dl$.  
A solution of System~(\ref{vdP}) in this subspace with $f \in \Dl$ is shown in Figure~\ref{C4ab-a-bFig}(i).
The nonzero parameters are $\alpha = 1$ and $\delta = 1$.
The reduced system is two copies of an identical, uncoupled oscillator
\begin{align*} 
\ddot u_1 &=  (1-u_1^2) \dot u_1 - 3 u_1   \\
\ddot u_2 &=  (1-u_2^2) \dot u_2 - 3 u_2 .
\end{align*}
Since $u_1$ and $u_2$ are decoupled in the reduced equations, and each equation has an attracting limit cycle,
the four trajectories in the phase space, $\dot u_i$ vs $u_i$, all lie on top of each other.
The period of each oscillator is identical and the solution has an arbitrary phase shift between the two anti-synchronized pairs.
This behavior was described by Alexander and Auchmuty in \cite{Alexander&Auchmuty}.

A solution in this same subspace with $f \in \Dodd \setminus \Dl$ is shown in Figure~\ref{C4ab-a-bFig}(ii)
The nonzero parameters are $\alpha = 1$, $\delta = 1$ and $\eps = 0.1$.
The reduced system is now
\begin{align*} 
\ddot u_1 &= (1-u_1^2) \dot u_1 - 3 u_1 + 0.1( (u_2-u_1)^3 +(-u_2-u_1)^3) \\
\ddot u_2 &= (1-u_2^2) \dot u_2 - 3 u_2 + 0.1( (u_1-u_2)^3 + (-u_1-u_2)^3 ).
\end{align*}
Now the two equations are coupled, and the solution shown in Figure~\ref{C4ab-a-bFig}(ii) is not an attractor.  The solution is evolving
toward an attractor in the subspace $\{ (a, a, -a, -a) \mid a \in \R^2\}$.
\end{example}

\section{Conclusion}



Our initial experience
as a collaborative group concerned solutions and numerical approximations of solutions to
semilinear elliptic PDE and PdE (partial difference equations), e.g., Example~\ref{semilinear}, \cite{NSS3, NSS5}.
In such works we not only observe invariant subspaces but use them to make our Newton's method-based algorithms more robust and efficient.
For many domains and nonlinearities,
the invariant subspaces are essentially all fixed point subspaces, which arise from symmetry.
By analyzing the symmetries of eigenfunctions of the linear elliptic part of the operator,
we are able to build bifurcation digraphs (labeled lattices of isotropy subgroups).
The digraphs have proven to be an efficient and effective tool for finding and interpreting many solutions to many of our types of
nonlinear problems.

Missing from this understanding was any theory explaining
what we called anomalous invariant subspaces (AIS),
invariant subspaces which are not fixed point subspaces.
For some graphs in \cite{NSS3}, in particular Sierpinski pre-gaskets,
the number of AIS can in fact dominate the number of fixed point subspaces,
in which cases our algorithms as currently implemented again have difficulties with robustness and efficiency.
The current work is a first step toward understanding bifurcations 
from one invariant subspace to another in cases that are not standard symmetry-breaking bifurcations.

\bibliographystyle{plain}
\bibliography{nss6}

\end{document}